\documentclass{amsart}
\usepackage{amsmath}
\usepackage{amsfonts}
\usepackage{amsthm, upref}
\usepackage{graphicx}
\usepackage[usenames, dvipsnames]{color}

\newcommand{\comment}[1]{}
\newcommand{\eq}{\begin{equation}}
\newcommand{\en}{\end{equation}}
\newcommand{\rr}{\mathbb{R}}

\newcommand{\norm}[1]{\left\lVert #1 \right\rVert}
\newcommand{\abs}[1]{\left\lvert #1 \right\rvert}

\newcommand{\iprod}[1]{\left\langle #1 \right\rangle }

\newcommand{\newR}{\mathfrak{R}}
\newcommand{\ev}{\mathbb E}  

\newcommand{\gen}{\mathcal{L}}
\newcommand{\gendual}{\hat{\mathcal{L}}}

\newcommand{\ep}{\hfill $\Box$}

\begin{document}

\theoremstyle{plain}
\newtheorem{thm}{Theorem}
\newtheorem{lemma}[thm]{Lemma}
\newtheorem{prop}[thm]{Proposition}
\newtheorem{cor}[thm]{Corollary}

\theoremstyle{definition}
\newtheorem{defn}{Definition}
\newtheorem{asmp}{Assumption}
\newtheorem{notn}{Notation}
\newtheorem{prb}{Problem}

\theoremstyle{remark}
\newtheorem{rmk}{Remark}
\newtheorem{exm}{Example}
\newtheorem{clm}{Claim}

\title[Rate of convergence]{Convergence rates for rank-based models with applications to portfolio theory}

\author{Tomoyuki Ichiba, Soumik Pal, and Mykhaylo Shkolnikov}
\address{Department of Statistics and Applied Probability \\ University of California \\ Santa Barbara, CA 93106}
\email{ichiba@pstat.ucsb.edu}
\address{Department of Mathematics\\ University of Washington\\ Seattle, WA 98195}
\email{soumik@u.washington.edu}
\address{Department of Mathematics\\ Stanford University\\ Stanford, CA 94305}
\email{mshkolni@math.stanford.edu}



\thanks{This research is partially supported by NSF grants DMS-1007563 and DMS-0806211}

\date{\today}

\begin{abstract} 
We determine rates of convergence of rank-based interacting diffusions and semimartingale reflecting Brownian motions to equilibrium. Convergence rate for the total variation metric is derived using Lyapunov functions. Sharp fluctuations of additive functionals are obtained using Transportation Cost-Information inequalities for Markov processes. We work out various applications to the rank-based abstract equity markets used in Stochastic Portfolio Theory. For example, we produce quantitative bounds, including constants, for fluctuations of market weights and occupation times of various ranks for individual coordinates. Another important application is the comparison of performance between symmetric functionally generated portfolios and the market portfolio. This produces estimates of probabilities of ``beating the market''.       
\end{abstract}

\maketitle

\section{introduction}

Informally a rank-based model is a multidimensional Markov process whose instantaneous dynamics is a function of the order in which the coordinates can be \textit{ranked}. When the process is a diffusion its movement can be described by a stochastic differential equation (SDE) as follows.  Let $n$ be the dimension of the process. Let $\delta_i$ and $\sigma_i$, $1\le i\le n$, be two finite collections of real and positive real constants, respectively. For any vector $x\in \rr^n$, let $x_{(1)} \le x_{(2)} \le \ldots \le x_{(n)}$ denote its ranked coordinates in increasing order. Consider the following system of stochastic differential equations:  For each $i=1,2,\ldots,n$ we have
\begin{equation}\label{eq1}
dX_i(t) = \left(\sum_{j=1}^n \delta_j\cdot1_{\left\{X_i(t)=X_{(j)}(t)\right\}}\right) dt + \left(\sum_{j=1}^n \sigma_j\cdot1_{\left\{X_i(t)=X_{(j)}(t)\right\}}\right) dW_i(t).
\end{equation}
Here $W=(W_1, \ldots, W_n)$ is a system of jointly independent one-dimensional standard Brownian motions. Assumptions and conditions guaranteeing the existence of such processes will be discussed later in the text. 

Different versions of the particle system in (\ref{eq1}) have been considered in several recent articles. Among the more recent ones, see Banner, Fernholz and Karatzas \cite{BFK}, Banner and Ghomrasni \cite{BG}, McKean and Shepp \cite{sheppmckean}, Pal and Pitman \cite{pp}, Jourdain and Malrieu \cite{joumal}, Chatterjee and Pal \cite{CP,chatpal2}, Ichiba and Karatzas \cite{IK}, Ichiba et al.~\cite{IPBKF}, Pal and Shkolnikov \cite{PS}, and Shkolnikov \cite{sh,sh2}. Related discrete time processes are studied in the context of competing particle systems by Arguin and Aizenman \cite{AA}, Ruzmaikina and Aizenman \cite{ruzaizenman}, Shkolnikov \cite{shkol}, and R\'acz \cite{racz}. We refer the reader to the above articles for the full list of applications of such processes (keywords: McKean-Vlasov equations, competing particles, dynamic models of spin glasses, models of equity markets, and queueing theory).

Processes satisfying \eqref{eq1} do not have any equilibrium. For example consider the average of all the coordinates. The resulting process is a linear diffusion with a constant drift and a constant diffusion coefficient. This process clearly has no equilibrium. The interesting long-term asymptotics arise when one considers the centered process by subtracting the average from each of the coordinates.  Under appropriate conditions (see below), the centered coordinate process is positively recurrent in the neighborhoods of the origin and has a unique stationary distribution. 

Another interesting and useful feature of this model is the process of spacings:
\[
Y_i(t) = X_{(i+1)}(t) - X_{(i)}(t), \qquad i=1,2,\ldots, n-1.
\]
The vector $Y=(Y_1, Y_2, \ldots, Y_{n-1})$ has the law of a semimartingale reflecting Brownian motion (SRBM) in the positive orthant 
$(\rr_{+})^{n-1}$ 
with a constant drift and diffusion coefficient and a constant oblique direction of reflection on each face of the orthant.  
This is an important class of processes that arise as a heavy-traffic limit of queues. Please see the survey by Williams \cite{W2}.

A particularly interesting example of the rank-based models is the so-called Atlas model in which one takes $\sigma_i = \sigma$ for all $i$ and $\delta_i=0$ for all $i > 1$. At any moment, the smallest particle gets an upward drift of $\delta_1=\delta >0$ while the others are locally independent Brownian motions. The asymmetry of the interaction among the particles sets it apart from the usual models of colliding particles (see, e.g., Harris \cite{harris65}, Swanson \cite{Sw}). In discrete time, this model can be compared with the asymmetric exclusion process or the Hammersley-Aldous-Diaconis process. 
 
In this paper we shall determine exponential rates of convergence to the equilibrium for rank-based models and the derived reflecting Brownian motions. We also obtain sharp Gaussian fluctuations around the equilibrium mean for additive functionals of these processes. The main difficulty one encounters with these processes is the lack of smoothness in its drift and diffusion parameters. Classical theorems often cannot be used and novel methods have to be invented. Our approach is a combination of several tools: an extended construction of Lyapunov functions for SRBM introduced in Dupuis and Williams \cite{DW}, recent advances in Transportation Cost-Information inequalities as in Bakry et al.~\cite{BBCG} and Guillin et al.~\cite{GLWY} and classical Poincar\'e inequalities satisfied by the associated Dirichlet forms. 

We should add that geometric ergodicity for certain recurrent SRBMs has been shown in a recent paper by Budhiraja and Lee \cite{BL}. However, these results lack constants and the proofs cannot be modified to yield explicit estimates. Our results have explicit constants wherever possible. We also give a shorter and simpler proof of geometric ergodicity itself.   

One of our primary motivations is to solve certain open problems described in the survey by Fernholz and Karatzas \cite{FI}. These are related to the area of Stochastic Portfolio Theory to which we provide a very brief introduction.

\subsection{A brief introduction to stochastic portfolio theory}

Stochastic Portfolio Theory (Fernholz \cite{F}, Fernholz \& Karatzas \cite{FI}) is a descriptive theory of equity market models (i.e., a dynamical model of total wealth that various companies raise through their stocks) which aims to be compatible with data on long-term market structure. This is a departure from the usual Black-Scholes models that are normative and are not supported by data. One significant difference between the two models is that while the Black-Scholes model assumes the principle of \text{no-arbitrage} in all its formulas, models in SPT, in fact, try to uncover arbitrages.   

\begin{figure}[t]
\begin{center}
\includegraphics[height=3.2in]{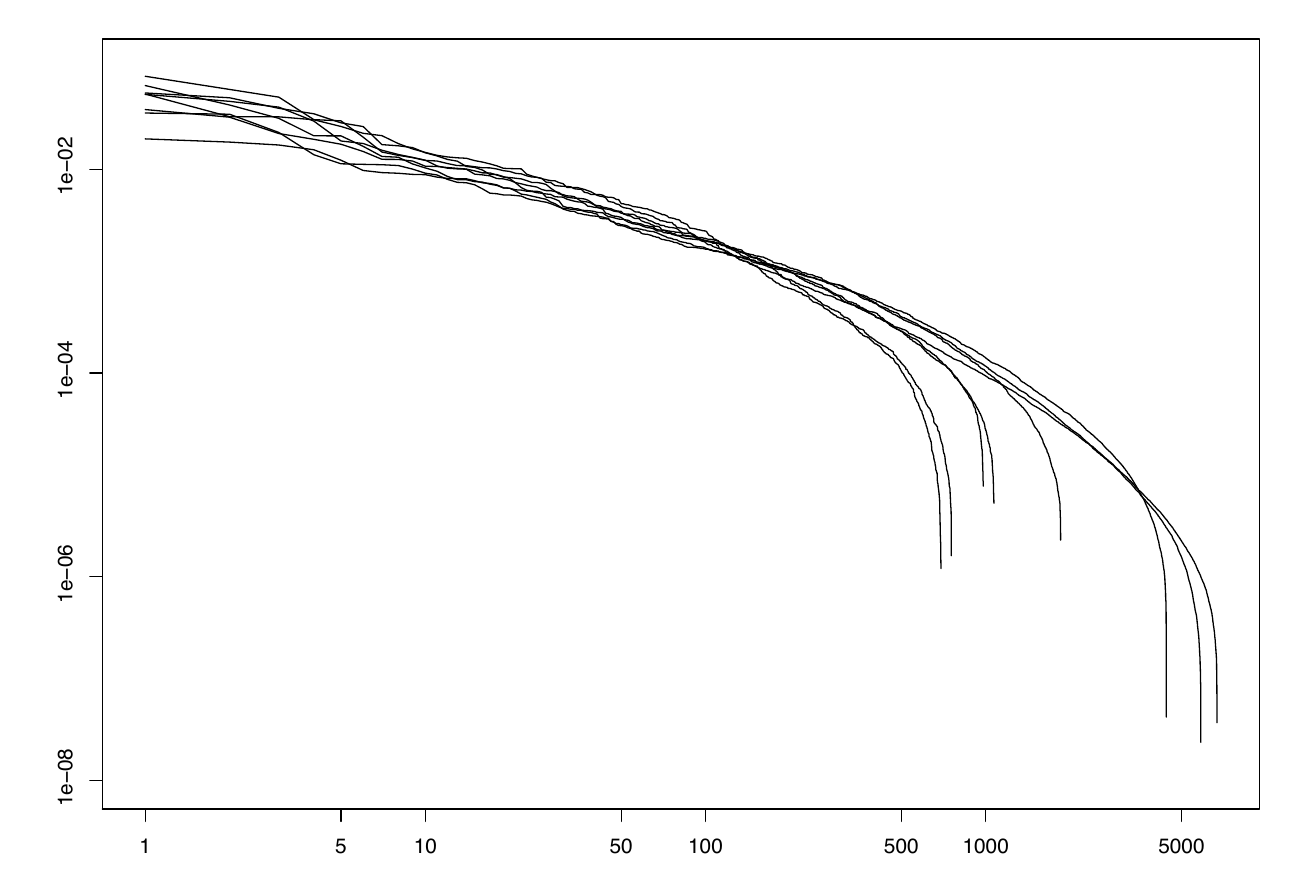}
\caption{Capital distribution curves: 1929-1999}
\label{capdist}
\end{center}
\end{figure}

The importance of the rank-based models stems from the fact that they match the data of the \textit{capital distribution curve}. To explain this, suppose index $i$ denotes a typical company listed in any of the major U.S. stock exchanges. Let $S_i(t)$ be its market capital (i.e., number of shares $\times$ price of each share) at time $t$. Let $S(t)=(S_1, S_2, \ldots, S_n)(t)$ denote the vector of market capitals of all such companies. Here $n$ is of the order of several thousands. We define the market weights as
\[
\mu_{i}(t) = \frac{S_i(t)}{\sum_{j=1}^n S_j(t)}, \quad i=1,2,\ldots,n.
\]
It has been observed for over \textit{eight decades} that if we arrange the market weights in decreasing order then they display a power-law decay. In other words, the $i$th largest market weight is proportional to $i^{-\alpha}$. See Figure \ref{capdist} which plots $\log \mu_{(n+1-i)}(t)$ versus $\log i$ for one instance of $t$ per decade. These curves are called capital distribution curves. It is known (see \cite{BFK}, \cite{CP}) that if we define $X_i=\log S_i$ to be a rank-based model, then, under appropriate assumptions, such a power law decay can be proved in equilibrium for large $n$.  

However, the stability of the market weights over time is more mysterious. It is clear that market parameters do not remain constant. In fact, the number of listed companies have grown ten fold. The reason for this stability has 
three reasons: (i) the limiting shape is independent of the parameters satisfying certain weak conditions, (ii) a fast return to equilibrium when perturbed in the parameters or dimension, (iii) tame fluctuations during the period that it takes the process to mix. (i) has been shown in \cite{CP}. In \cite{PS}, the authors have proved a fluctuation of order $\sqrt{T}$ for the shape of the market weights over intervals of time of length $T$. This shows (iii). In this article we take up (ii) and show an exponential rate of convergence for market weights.   

Another important topic of analysis is the performance of portfolios when the market is modeled as rank-based. A portfolio $(\pi_1(t), \ldots, \pi_n(t))$ is a random process that takes values in the simplex $\{ x\in \rr^n:\; x_i \ge 0 \; \text{and}\; \sum_i x_i=1\}$. The quantity $\pi_i(t)$ represents the proportion of wealth an investor invests in stocks of company $i$ at time $t$. When $\pi\equiv \mu$, the market weights, the portfolio is called the market portfolio. The latter is of central importance since the market portfolio determines various index funds (i.e., portfolio tracking an index such as S\&P 500). An important question in practice is the relative performance of portfolios when compared to the market portfolio. 

Let $V^{\pi}(t)$ denote the wealth of an investor who has followed portfolio $\pi$ with an initial investment of $\$ 1$. Often 
$\pi= {\mathbf H}(\mu)$ 
is a function of the market weights. Such portfolios are called functionally generated (see the survey \cite{FI}). We focus on estimating tail probabilities of the quantity $V^{\pi}(t)/V^{\mu}(t)$ and its reciprocal for a functionally generated portfolio $\pi$. These are probabilities that the portfolio $\pi$ will \textit{beat the market} which determine its attractiveness to investors.   

This is closely related to finding relative arbitrage opportunities. According to the definition in \cite{FI}, the rank-based models are not \textit{diverse}, i.e., the market entropy process is not bounded from below. This implies that the entropy weighted portfolio (see Section \ref{portfolio} below) does not necessarily provide a long-term arbitrage opportunity over the market portfolio. Thus a natural question is an estimate of the tail probability of the ratio of value processes between the functionally generated portfolio and the market portfolio.

\subsection{Main results}

Define the set $I$ as $\{1,\dots,n\}$ and consider the SDE in \eqref{eq1}: 
\begin{equation}\label{ranksde}
dX_i(t) = \sum_{j\in I} \delta_j\cdot1_{\left\{X_i(t)=X_{(j)}(t)\right\}} dt + \sum_{j\in I} \sigma_j\cdot1_{\left\{X_i(t)=X_{(j)}(t)\right\}} dW_i(t),\quad i\in I.
\end{equation}
The weak solution of (\ref{ranksde}) exists and it is unique in law (e.g., \cite{BFK}, \cite{BP}).  

We assume the following constraint on the drift constants: for $1\le k \le n-1$, the constants
\eq\label{deltaassump}
\alpha_k := 2\sum_{i=1}^k \left( \delta_i - \overline{\delta}  \right) \quad \text{are strictly positive},
\en
where $\overline\delta := \sum_{j \in I} \delta_{j} / n$ is the average drift.
This is a stability criterion that guarantees the existence of an invariant distribution for spacings 
or ``gaps'' 
(see \cite{IPBKF}). 

For some of our results we will also assume the condition
\eq\label{sigmaassump}
\sigma_2^2-\sigma_1^2=\dots=\sigma_n^2-\sigma_{n-1}^2
\en
on the diffusion coefficients.
This is an assumption that is directly motivated by data on volatility. See Figure 13.6 in \cite{FI} which shows estimated volatilities to be almost linearly decreasing with rank.   

Let $\nu$ denote the law on $\rr^{n-1}$ of an $(n-1)$-dimensional random vector with independent coordinates, where the $k$-th coordinate is distributed according to the Exponential distribution of parameter 
\[
\tilde{\alpha}_k:=2\alpha_k(\sigma_k^2+\sigma_{k+1}^2)^{-1}. 
\]
It is shown in Corollary 2 of \cite{IPBKF} that under \eqref{deltaassump} the distribution of the spacing system $(Y(t) := (Y_{1}(t), \ldots , Y_{n-1}(t)), t \ge 0)\, $ consisting of  
\[
Y_{j} (t) := X_{(j+1)}(t) - X_{(j)}(t), \quad  j = 1, \ldots , n-1, 
\]
converges in total variation norm as $t\to \infty$ to its unique stationary distribution. Moreover, under condition \eqref{sigmaassump} the latter is given by the measure $\nu$ defined above, as explained in section 5 of \cite{IPBKF}. 
We call a probability measure $\kappa $ supported by $(\rr_{+})^{n-1}$ a spacing distribution. The above $\nu$ is a spacing distribution. 

In this setting we obtain the following theorem.


\begin{thm}\label{mainadd}
Let condition \eqref{deltaassump} hold and suppose that $\sigma_i\equiv 1$, $i\in I$. Then for every initial spacing distribution $\kappa $ that is absolutely continuous and such that ${d\kappa}/{d\nu}$ is in 
$\mathbf{L}^2((\rr_{+})^{n-1}; \nu)$, 
and for every bounded function $u$ of the spacings with 
$\nu(u):= \int_{(\rr_{+})^{n-1}} u\; d \nu = 0$
and 
$\text{Var}_\nu(u) := \int_{(\rr_{+})^{n-1}} u^{2} d \nu = \sigma^2$, 
and $t,r,\epsilon >0$, the following estimate holds:
\eq\label{maintailbound}
\begin{split}
P&\left( \frac{1}{t} \int_0^t u(Y(s)) ds \ge r  \right) \le\\
&\norm{\frac{d\kappa}{d\nu}}_2 \exp\left[ -\frac{t}{\beta} \max\left( \frac{r^2}{\delta^2(u)}, 4\epsilon(\epsilon+ \sigma^2)\left( \sqrt{1 + \frac{r^2}{2\epsilon(\epsilon + \sigma^2)^2 \norm{u}_\infty^2}} -1   \right) \right)  \right].
\end{split}
\en
Here $\delta(u) := \sup\abs{u(x) - u(y)}$ is the range of $u$ and 
\eq\label{whatisbeta}
\beta=\frac{4 \lambda_{n}}{\min_{1\leq k\leq n-1} \tilde{\alpha}^2_k},
\en
where $\lambda_n$ is given in \eqref{lambda_n} below. 
\end{thm}

\begin{rmk}
Notice the effect of various parameters in the expression \eqref{maintailbound}. It is exponential in $t$ showing a geometric rate of convergence. It is Gaussian in $r$, when $r$ is large, which shows a strong concentration around the mean. This is evident in real data on additive functionals such as \textit{cumulative excess growth rate} as can be seen in Figure 11.3 in \cite{FI}.
\end{rmk}


\begin{thm}\label{skewadd} 
Let conditions \eqref{deltaassump} and \eqref{sigmaassump} be satisfied. 
\begin{enumerate}
\item[(i)] For every function $u\in\mathbf{L}^2((\rr_+)^{n-1};\nu)$ for which 
$\nu(u) = 0$, $\|\nabla u\|\in\mathbf{L}^2((\rr_+)^{n-1};\nu)$ and the conditions of Lemma \ref{poincarevariance} below hold, we obtain a variance bound
\[
\ev^{\nu} \Big[\Big( \frac{1}{t} \int^{t}_{0} u(Y(s)) d s \Big)^{2}\Big] = O (t^{-1}). 
\]
\item[(ii)] Furthermore, for every $0<\varepsilon<2$ and every initial spacing distribution $\kappa$ that is absolutely continuous with respect to $\nu$ and such that $d\kappa / d \nu$ belongs to $\mathbf{L}^{2/\varepsilon} ((\rr_{+})^{n-1}; \nu)$, the bound of part (i) is modified to
\[
\ev^{\kappa} \Big[ \Big ( \frac{1}{t} \int^{t}_{0} u(Y(s)) d s \Big)^{2-\varepsilon}\Big]= O(t^{-1}).
\]
\end{enumerate}
\end{thm}

\begin{rmk}
Under additional assumptions on the function $u$ in Theorem \ref{skewadd}, we obtain bounds of order $t^{-p/2}$ on the $p$-th moment of the corresponding additive functional, and with more assumptions also the Gaussian concentration property for the latter (see Lemma \ref{neumannlemma} below for the details). Moreover, it is clear from the proof that the result of Theorem \ref{mainadd} applies to any recurrent normally \textit{reflected Brownian motion} in a polyhedral domain with constant drift vector and identity covariance matrix. Similarly, the result of Theorem \ref{skewadd} holds for any recurrent \textit{semimartingale reflected Brownian motion} in a polyhedral domain with skew-symmetric data in the sense of \cite{W}.    
\end{rmk}

In the absence of condition \eqref{sigmaassump} we view the spacings process as an SRBM in the sense of \cite{DW}. 
The general theory of SRBMs is quite tricky and depends heavily of properties of Skorokhod maps. We will use the conditions of Theorem 2.6 in \cite{DW} which yield the existence of a unique invariant distribution. In technical terms these conditions are (i) a \textit{completely}-$\mathcal{S}$ condition on the reflection matrix, and (ii) the statement that all solutions of an associated deterministic Skorokhod problem are attracted to the origin.  
We reprove the geometric ergodicity result of Budhiraja and Lee \cite{BL} with a much shorter and simpler proof. 

\begin{thm}\label{mainlyapunov}
Let $Y$ be an SRBM satisfying the conditions of Theorem 2.6 in \cite{DW}. Denote by $P_t(y,.)$, $y\in(\rr_+)^{n-1}$, $t\geq0$ the transition probabilities of $Y$ and let $\|.\|_{TV}$ be the total variation norm. Then there exists a measurable function $M:\;(\rr_+)^{n-1}\rightarrow[0,\infty)$ and a constant $0<\zeta<1$ such that it holds
\eq
\|P_t(y,.)-\nu\|_{TV}\leq M(y)\zeta^t
\en  
for all $y\in(\rr_+)^{n-1}$, $t\geq0$, where $\nu$ is the unique invariant distribution of $Y$.
\end{thm}

\subsection{Outline} The article is organized as follows. Section \ref{skewsymmetry} establishes rates of convergence for SRBMs and their additive functionals under the condition of skew-symmetry. In Section \ref{general} we prove exponential rate of convergence for a general SRBM under a technical condition on the Skorokhod map. Section \ref{portfolio} establishes comparison bounds on the performance of functionally generated portfolios with respect to the market portfolio. Finally, in Section \ref{marketweightsetc} we give estimates on the fluctuations of the market weights and of the time spent by a given market weight in a given rank.

\section{Convergence to equilibrium under the skew-symmetry condition}\label{skewsymmetry}

To be ready to prove Theorem \ref{mainadd} we start with some preliminaries. In the first subsection of this section we recall the construction of certain normally reflected Brownian motions with constant drift. In the following subsection we explain how the latter are related to the spacings process in the case that $\sigma_i=1$, $i\in I$ and give the proof of Theorem \ref{mainadd} in this case. Finally, in the last subsection of this section we extend our  approach to the setting of Theorem \ref{mainadd}. 

\subsection{Normally reflected Brownian motion with constant drift} 

Fix a dimension $n \in \mathbb{N}$, define the open wedge 
\[
H= \left\{ x\in \rr^n:\; \sum_{i=1}^n x_i > 0, \quad \text{and} \quad x_1 < x_2 < \ldots < x_n   \right\}
\] 
and let $\gamma=(\gamma_1, \ldots, \gamma_n)$ be a given vector of constants. 

We consider a normally reflected Brownian motion (RBM) on $\overline{H}$ with a constant drift vector $\gamma$, an identity covariance matrix, and normal reflection. This process can be obtained by the following Dirichlet form construction as described in section 2 of Burdzy et al.~\cite{BCP}. 

First, define the positive measure
\[
m(dx) := 1_{H}(x) \exp(2\langle\gamma,x\rangle)dx.
\]
Next, let $\mathcal{F}$ be the collection of continuously differentiable functions on $H$ such that both the function and its gradient are square integrable with respect to $m$. That is,
\eq\label{domaindirichlet}
\mathcal{F}= \left\{  u\in \mathbf{L}^2(H;m)\;: \quad \nabla u \in \mathbf{L}^2(H;m)   \right\}.
\en
Over this domain we define the symmetric Dirichlet form 
\[
\mathcal{E}(u,v) :=\frac{1}{2} \int_H \iprod{\nabla u(x), \nabla v(x)  } m(dx), \quad \text{for}\; u,v \in \mathcal{F}.
\]
It can be shown (see \cite{BCP}) that $(\mathcal{E}, \mathcal{F})$ is a regular Dirichlet form on $\mathbf{L}^2(\overline{H}; m)$ in the sense that $C_c(\overline{H})\cap \mathcal{F}$ is dense both in $C_c(\overline{H})$ in the uniform norm and in $\mathcal{F}$ with respect to the Hilbert norm 
\[
\sqrt{\mathcal{E}_1(u,u)}:= \sqrt{\mathcal{E}(u,u) + \iprod{u,u}_{\mathbf{L}^2(H;m)}}.
\]
Hereby, as usual, $C_c(\overline{H})$ denotes the space of continuous functions with compact support in $\overline{H}$.

A strongly continuous Markov process $X$ is called an RBM on $H$ with a constant drift $\gamma$, if it is symmetric with respect to the measure $m$ and the associated Dirichlet form is given by $(\mathcal{E}, \mathcal{F})$ in the sense
\[
\begin{split}
\mathcal{F}&=\left\{ u\in \mathbf{L}^2(H;m):\; \lim_{t\rightarrow 0}\frac{1}{t} \iprod{u- P_tu, u}_{\mathbf{L}^2(H;m)} < \infty  \right\}\\
\mathcal{E}(u,v) &= \lim_{t\rightarrow 0}\frac{1}{t} \iprod{u - P_tu, v}_{\mathbf{L}^2(H;m)}, \quad \text{for}\; u,v \in \mathcal{F}.
\end{split}
\]
Here $(P_t)_{t\geq0}$ refers to the transition semigroup of the process. 

The function $h(x)= 1_{H}(x) \exp(2\iprod{\gamma,x})$ is the density of an invariant measure for $X$. When $h$ is integrable with respect to the Lebesgue measure on $H$, the normalized measure (which we will continue to denote by $m$) constitutes the unique invariant distribution of the RBM. 

It is known that this process is unique in law and admits a semimartingale decomposition, given by the solution of a deterministic Skorokhod problem applied to the path of a Brownian motion with drift vector $\gamma$. In particular, the process has a local time at the boundary and, started from any point in $H$, almost surely never visits the origin \cite{W}.  

\subsection{Proof of Theorem \ref{mainadd}}

We now return to the rank-based model described in the introduction and assume throughout this subsection that $\sigma_i=1$, $i\in I$. As explained in the proof of Theorem 8 in \cite{pp}, the law of a suitably shifted vector of ordered processes in this model is identical to that of an RBM in 
$\overline{H}$. 
More specifically, let \mbox{\boldmath $\beta$} denote an independent one-dimensional RBM with a negative drift $-\theta$ ($\theta > 0$). Define
\[
\tilde{Y}_i(t) := X_{(i)}(t) - \overline{X}(t) + \frac{\text{\boldmath $\beta$}(t)}{\sqrt{n}}, \qquad i=1,2,\ldots,n.
\]
Note that $\sum_{i \in I} \tilde{Y}_i(t) = \sqrt{n} \text{\boldmath $\beta$} (t) \ge 0$ for $ t \ge 0$.  
In addition, for every $t\geq0$ the spacings between the components of the vector 
$\tilde{Y}(t)$ are the same as the original spacings, and, moreover, the process 
$\tilde{Y}$ is an RBM in $H$, as described above, with a drift vector $\gamma$ determined by
\[
\iprod{\gamma, \tilde{y}}=\sum_{i=1}^n \left( \delta_i - \overline{\delta} - \frac{\theta}{\sqrt{n}} \right) \tilde{y}_i = - \frac{1}{2} \sum_{k=1}^{n-1} \alpha_k \left( \tilde{y}_{k+1} - \tilde{y}_k \right) - \frac{\theta}{\sqrt{n}} \sum_{i=1}^n \tilde{y}_i,\qquad \tilde{y}\in H.
\]

It follows from \cite[Lemma 9]{pp} that the corresponding invariant measure $m$ is integrable with respect to the Lebesgue measure on $H$ and that, under its normalized version, the law of the spacings $(\tilde{y}_{k+1} - \tilde{y}_k,\; k=1,\ldots, n-1)$ is given by a product of 
Exponential distributions with parameters $\alpha_k,\;k=1,\ldots,n-1$.

\bigskip

For any $\tilde{y}\in H$, we make the transformation
\[
y_k = \tilde{y}_{k+1} - \tilde{y}_k,\quad k=1,\ldots, n-1,
\]
which maps configurations of ordered particles to the corresponding vectors of spacings. Next, we let $T$ be the unique linear extension of this map to the whole of $\rr^n$. The {\it push-forward} of the measure $m$ by $T$, which is a probability measure on the positive orthant $(\rr_+)^{n-1}$, will be denoted by $\nu$ and is given by the product of Exponential distributions with parameters $\alpha_k,\;k=1,\ldots,n-1$. 

Consider any function $f: (\rr_+)^{n-1}\rightarrow\rr$ which is continuously differentiable. The gradients $\nabla_y f$ on $(\rr_+)^{n-1}$ and $\nabla_{\tilde{y}}(f \circ T)$ on $H$ satisfy the simple linear equation
\[
\frac{\partial (f\circ T)}{\partial \tilde{y}_i} = \sum_{j=1}^{n-1} \frac{\partial f}{\partial y_j} \frac{\partial y_j}{ \partial \tilde{y}_i}=  \sum_{j=1}^{n-1} a_{ij}\frac{\partial f}{\partial y_j},
\]
where $A=(a_{ij})$ is an $n \times (n-1)$ matrix given by
\[
a_{ij}= \begin{cases}
-1,&\text{if}\quad 1\le j=i \le n-1,\\
+1,&\text{if}\quad 1\le j=i-1\le n-1,\\
0,&\text{otherwise}.
\end{cases}
\]

Moreover, it is clear that $f$ and its gradient are in $\mathbf{L}^2((\rr_+)^{n-1};\nu)$ if and only if $g:=f\circ T$ and the gradient of $g$ are in $\mathbf{L}^2(H;m)$. In fact, we have the following bound:
\eq\label{DFcompare}
\int_{(\rr_+)^{n-1}} \norm{\nabla_{y} f}^2 d\nu  \le \lambda_n \int_H \norm{\nabla_{\tilde{y}} g}^2 dm. 
\en
Here, $1/\lambda_n$ is the minimal eigenvalue of the matrix $A'A$, with $A'$ denoting the transpose of $A$. To compute the latter eigenvalue explicitly, we note that $2\cdot Id-A'A$ is the $(n-1)\times(n-1)$ tridiagonal matrix which has zeros on the diagonal and ones next to the diagonal, with $Id$ being the $(n-1)\times(n-1)$ identity matrix. It is well-known (see e.g. \cite{KO} and the references there) that the eigenvalues of a tridiagonal matrix of this type are given by $2\cos\frac{k\pi}{n}$, $k=1,\dots,n-1$, so that the largest eigenvalue of $2\cdot Id-A'A$ is given by $2\cos\frac{\pi}{n}$. Thus, 
\eq\label{lambda_n}
\lambda_n=\frac{1}{2-2\cos\frac{\pi}{n}}.    
\en 
In particular, we see that $\lambda_n$ grows quadratically in $n$. 
 
\bigskip

Now, recall that a probability measure $\mu$ on $\rr^n$ is said to satisfy the Poincar\'e inequality with the Poincar\'e constant $C_P>0$ if for every continuously differentiable function $f:\rr^n \rightarrow \rr$ it holds
\eq\label{whatispoincare}
\text{Var}_\mu(f) = \int_{\rr^{n}} f^{2} d\mu - \Big( \int_{\rr^{n}} f d\mu \Big)^{2}\le C_P \int_{\rr^{n}} \norm{\nabla f}^2 d\mu,
\en
where $\text{Var}_\mu(f)$ stands for the variance of $f$ under $\mu$. 

We proceed by recalling the fact that the Exponential distribution with parameter $1$ satisfies the Poincar\'e inequality with a constant $C_P=4$. This well-known result can be found in several sources, such as the book by Ledoux  
(Lemma 5.1 of \cite{L}) 
or the article by Barthe and Wolff \cite{BW}, who also prove several other interesting results about the Gamma family. 

It follows by a simple scaling argument that the Exponential distribution with a parameter $\lambda>0$ satisfies the Poincar\'e inequality with the Poincar\'e constant $C_P=4/\lambda^2$. In addition, we use the fact that the Poincar\'e inequality has the following tensorization property: if $\mu_1,\ldots,\mu_n$ are probability measures on Euclidean spaces satisfying the Poincar\'e inequality with the constants $C_1, \ldots, C_n$, then the product measure $\mu_1\times\ldots\times\mu_n$ satisfies the Poincar\'e inequality on the corresponding product space with the Poincar\'e constant $\max_{i=1,\dots,n} C_i$.

The above considerations lead to the following lemma. 

\begin{lemma}\label{exppoincare} The measure $\nu$ defined in the beginning of this subsection satisfies the Poincar\'e inequality with the Poincar\'e constant
\[
C_\nu:=4\cdot\max_{1\le k \le n-1} \alpha_k^{-2} = 4\cdot\Big(\min_{1 \le k \le n-1} \alpha_{k}^{2}\,\,\,\Big)^{-1}.
\]
\end{lemma}

Combining Lemma \ref{exppoincare} with the comparison bound in \eqref{DFcompare}
and the equation \eqref{lambda_n}, 
we obtain for all functions $f\in\mathcal{F}$ the estimate
\eq\label{poincare_inf}
\text{Var}_m \left(f \right) \le \frac{1}{2-2\cos\frac{\pi}{n}}\cdot\frac{4}{\min_{1\le k\le n-1} \alpha^2_k}\cdot\mathcal{E}(f,f)=:\beta \mathcal{E}(f,f).
\en
Here, the pair $(\mathcal{E},\mathcal{F})$ is defined for the RBM $\tilde{Y}$ according to the previous subsection. 

\bigskip

We can now finish the proof of Theorem \ref{mainadd} in the case that $\sigma_i=1$, $i\in I$. Indeed, applying Theorem 3.1 from Guillin et al. \cite{GLWY}, we obtain the estimate of Theorem \ref{mainadd} for any bounded measurable function of the particle configuration given by the components of the process $\tilde{Y}$. In particular, the statement of Theorem \ref{mainadd} holds for any bounded measurable function of the vector of spacings, as claimed.

\subsection{Proof of Theorem \ref{skewadd}}

We now give the proof of Theorem \ref{skewadd}. In this part we assume the condition \eqref{deltaassump} on the drift coefficients and the condition \eqref{sigmaassump} on the diffusion coefficients. The proof is broken down in several steps. 

\bigskip

\noindent\textbf{Step 1: Reduction to SRBM.} The ordered particle system corresponds to the following SDE:
\eq\label{ordered_part}
d X_{(i)}(t) = \delta_i dt + \sigma_i dB_i(t) +  \frac{1}{2}d L_{i-1,i}(t) - \frac{1}{2} d L_{i,i+1}(t),\quad i=1,\dots,n,
\en
where $L_{i,i+1}$ refers to the local time of collisions between the $i$-th and the $(i+1)$-st ranked particles (see section 3 of \cite{BFK}). In other words, this is the semimartingale local time at zero for the process $X_{(i+1)}(t) - X_{(i)}(t)$, $t\geq0$ normalized according to the It\^o-Tanaka formula.

In the following we closely follow the analysis and linear algebra done in Ichiba and Karatzas \cite[Section 3.2.1]{IK}. The process $(X_{(1)},\dots,X_{(n)})$ is a semimartingale reflected Brownian motion (SRBM) in the sense of \cite{W}, with a constant drift vector, a constant diffusion matrix and normal reflection, taking values in the Euclidean closure of the wedge 
\[
\left\{(x_1,\dots,x_n)\in\rr^n:\;x_1\leq\ldots\leq x_n\right\}.
\]

Next, we consider the process of spacings $Y$, whose components are governed by the SDEs
\[
\begin{split}
d Y_i(t) = (\delta_{i+1} - \delta_i)dt + \sigma_{i+1} d B_{i+1}(t) - \sigma_i d B_i(t) \quad\;\;\\
 +  d L_{i, i+1}(t) - \frac{1}{2} d L_{i+1, i+2}(t) - \frac{1}{2} d L_{i-1,i}(t),
\end{split}
\]
$i=1,\dots,n-1$. It is easy to verify that $Y$ is also an SRBM in the sense of \cite{W}, taking values in the positive orthant $(\rr_+)^{n-1}$, and having a constant drift vector $\gamma$, a constant diffusion matrix $\Xi$ and a reflection matrix $R$. 

More specifically, the drift vector is given by $\gamma=(\delta_2-\delta_1,\dots,\delta_n-\delta_{n-1})$. The covariance matrix $\Xi= (\xi_{ij})_{1\le i,j \le n-1}$ is tridiagonal with components
\eq \label{matrixXi}
\xi_{ij} =\begin{cases}
\sigma_i^2 + \sigma_{i+1}^2,&\quad \text{when}\; 1\le i=j \le n-1,\\
-\sigma_i^2,& \quad \text{when}\; 2\le i=j+1 \le n,\\
-\sigma_{i+1}^2,&\quad \text{when}\; 1\le i =j-1 \le n-1,\\
0,&\quad \text{otherwise}. 
\end{cases}
\en
The reflection matrix $R=(r_{ij})_{1\le i,j \le n-1}$ is also tridiagonal and its components are given by
\eq\label{whatisr}
r_{ij} =\begin{cases}
1,&\quad \text{when}\; 1\le i=j \le n-1,\\
-\frac{1}{2},& \quad \text{when}\; 2\le i=j+1 \le n,\\
-\frac{1}{2},&\quad \text{when}\; 1\le i =j-1 \le n-1,\\
0,&\quad \text{otherwise}. 
\end{cases}
\en

Let $\Sigma$ denote a nondegenerate square root of the matrix $\Xi$, that is, an invertible $(n-1)\times(n-1)$ matrix which satisfies $\Sigma\Sigma'=\Xi$. Note hereby that the matrix $\Xi$ is stricly positive definite, since none of the processes $\sigma_{i+1}B_i(t)-\sigma_i B_i(t)$, $t\geq0$ for $i=1,\dots,n-1$ can be obtained as a linear combination of the other $n-2$ processes of the same type. 

Now, one can make the transformation $Z(t):=\Sigma^{-1} Y(t)$, $t\geq0$ to obtain an SRBM in a polyhedral domain $G$ given by the image of the orthant $(\rr_+)^{n-1}$ under $\Sigma^{-1}$. This new SRBM has a constant drift vector given by $\Sigma^{-1}\gamma$, identity covariance matrix, and a new reflection matrix $\newR:=\Sigma^{-1} R$.

\bigskip

Hence, for the sake of the proof of Theorem \ref{skewadd} we can restrict our attention to the case of an SRBM taking values in the closure of a polyhedral domain $G$ with data of the following form: 
\begin{enumerate}
\item[(i)] drift vector $\Sigma^{-1}\gamma$ and reflection matrix $\newR$, 
\item[(ii)] a cone $G$ of dimension $n-1$ given by the dual description:
\[
G=\left\{x:\; (\Sigma x)_i > 0,\;i=1,\dots,n-1 \right\}.
\] 
In particular, the $(n-1)\times(n-1)$ matrix $N$ whose columns are the unit normal vectors on the faces of $G$ is given by the normalized row vectors of $\Sigma$. In addition, the equality $\Xi=\Sigma \Sigma'$ implies that the norm of the $i$-th row of $\Sigma$ is equal to $\sqrt{\xi_{ii}}$. Thus, if we let $D$ denote the diagonal matrix comprised of the diagonal elements of $\Xi$, we have $N=\Sigma' D^{-1/2}$.

\item[(iii)] Due to the condition \eqref{sigmaassump} the reflection matrix $\newR$ admits the decomposition
\[
\newR = N + Q,
\]
where the diagonal elements of $N'Q$ are all zero. In fact, under \eqref{sigmaassump} the so-called \textit{skew-symmetry condition} 
\begin{equation}\label{skewsym}
N'Q + Q'N =0
\end{equation}
holds (see page 25 in \cite{IK} for details).
\end{enumerate}

Following Williams \cite{W} we call the process $Z$ an SRBM associated with the parameters $(N,Q,\Sigma^{-1}\gamma)$. Such an SRBM has the following semimartingale decomposition:
\eq\label{semidecomp}
Z(t) = Z(0) + \Sigma^{-1}\gamma t + W(t) + \left( N + Q \right) L(t),
\en
where $L$ is the vector of the accumulated local times at the faces of the boundary of $G$. 
\bigskip

\noindent\textbf{Step 2: Application of duality.} Next, we define $\tilde{\gamma}:= 2\left( I - N^{-1}Q  \right)^{-1}\Sigma^{-1}\gamma$ and let $\hat{\gamma}$ denote the vector $\tilde{\gamma}-\Sigma^{-1}\gamma$. The above considerations show that the following result from Williams \cite[Thm 1.2, Cor 1.1]{W} applies in our setting. 

\begin{thm}
Consider the measure $\rho$ on $G$ whose density with respect to the Lebesgue measure on $G$ is given by 
\[
\rho(x) = \exp(\langle\tilde{\gamma},x\rangle).
\] 
The SRBMs associated with $(N,Q,\Sigma^{-1}\gamma)$ and $(N,-Q, \hat{\gamma})$ are in duality relative to $\rho$ and $\rho$ is an invariant measure for these two processes. In particular, if $\rho$ is finite, the normalized measure 
(which we will also denote by $\rho$) 
is the unique stationary distribution for each of the two processes.  
\end{thm}

The duality being referred to is in the following sense: let $(P_t)_{t\geq0}$ and $(\hat P_t)_{t\geq0}$ denote the transition semigroups for the two SRBMs in the theorem. Then for all continuous functions $f$, $g$ with compact support in the closure of $G$ we have
\eq\label{ptduality}
\int_G (P_t f)(x) g(x) \rho(x)\;dx = \int_G f(x) (\hat P_t g)(x) \rho(x)\;dx. 
\en 

Now, let $C_c^\infty(G)$ be the collection of infinitely differentiable functions with compact support in the open cone $G$ and let $\gen$ and 
$\gendual$ denote the respective generators of the two SRBMs in the statement of the latter theorem on the domain $C_c^\infty(G) \subseteq \mathbf{L}^2(G;\rho)$. That is,  
\eq\label{whatisgen}
\gen = \langle \Sigma^{-1}\gamma, \nabla\rangle + \frac{1}{2} \Delta, \qquad \gendual = \langle \hat{\gamma}, \nabla\rangle + \frac{1}{2} \Delta. 
\en
From the said duality we conclude:
\[
\begin{split}
\iprod{-\gen f, g}_{\mathbf{L}^2(G;\rho)} = \iprod{-\gendual g, f}_{\mathbf{L}^2(G;\rho)}, \qquad f,g \in C_c^\infty(G).
\end{split}
\]

Next, for $f, g \in C_c^\infty(G)$, we define the symmetrized Dirichlet form $\mathcal{E}^\sigma$ with domain $C_c^\infty(G)$ by 
\eq\label{whatisesigma}
\begin{split}
\mathcal{E}^\sigma(f,g) &:= \frac{1}{2} \left[  \iprod{-\gen f, g}_{\mathbf{L}^2(G;\rho)}+ \iprod{-\gen g, f}_{\mathbf{L}^2(G;\rho)}\right] = \frac{1}{2}   \iprod{-\gendual g - \gen g, f}_{\mathbf{L}^2(G;\rho)}. 
\end{split}
\en
Moreover, using integration by parts we get
\eq\label{integbyparts}
\begin{split}
 \iprod{-\gendual g - \gen g, f}_{\mathbf{L}^2(G;\rho)} &= \int_G \left[ -\langle\tilde{\gamma},(\nabla g)(x)\rangle - (\Delta g)(x) \right] f(x) \exp(\langle\tilde{\gamma},x\rangle)\;dx\\
 &=\int_G \langle(\nabla g)(x),(\nabla f)(x)\rangle \exp(\langle\tilde{\gamma},x\rangle)\;dx.
\end{split}
\en
Thus,  
\eq\label{symmform}
\mathcal{E}^\sigma(f,g)  = \frac{1}{2} \int_G \langle(\nabla g)(x),(\nabla f)(x)\rangle \exp(\langle\tilde{\gamma},x\rangle)\;dx.
\en

In other words, $\mathcal{E}^\sigma$ is the pre-Dirichlet form for an SRBM with a constant drift vector, identity covariance matrix and normal reflection on the faces of $G$. It is known in such case (see the references in Burdzy et al. \cite{BCP}) that the Dirichlet form is closable in $\mathbf{L}^2(\overline G;\rho)$ and that the resulting closed extension is regular in a sense made precise in the previous subsection. Moreover, the domain of the Dirichlet form $\mathcal{E}^\sigma$ is given by \eqref{domaindirichlet}, with $H$ replaced by $G$ and $m$ replaced by $\rho$.
\bigskip

\noindent\textbf{Step 3: Forward-backward martingale decomposition.} Now to analyze additive functionals, note that the duality \eqref{ptduality} between $(P_t)_{t\geq0}$ and $(\hat{P}_t)_{t\geq0}$ can be expressed in the following way. Suppose $\{ Z(u), \; 0\le u \le t \}$ is an SRBM $(N, Q, \Sigma^{-1}\gamma)$ such that $Z(0) \sim \rho$ (which is assumed to be a Probability measure). Define the time-reversed process $\hat Z(s) = Z(t-s), \; 0\le s \le t$. Then $\{\hat Z(u), \; 0\le u \le t \}$ is an SRBM $(N, -Q, \hat \gamma)$ with $\hat Z(0) \sim \rho$. This follows since the time-reversed process is obviously Markov, and the duality equation \eqref{ptduality} determines its transition kernel. 

Now let $v:G \rightarrow \rr$ be a continuous function such that
\eq\label{condonv}
\norm{v}_\infty < \infty,\quad \text{and}\quad \int_G v(x) \rho(x) dx=0.
\en
Suppose $U:\overline{G} \rightarrow \rr$ is a twice continuously differentiable function (i.e., $U \in C^2(\overline{G})$) such that
\eq\label{uandv} 
-\left( \gendual + \gen \right)U(x)= 2v(x),\; x \in G, \quad \text{and}\quad  \left(\nabla U(x)\right)'n(x) =0, \; x\in \partial G. 
\en
Here $n(x)$ is the inward normal at the boundary point $x$. This vector is also one of the columns of the matrix $N$. 

In other words, $U$ is the solution of the Neumann problem:
\eq\label{neumannu}
- \iprod{\tilde\gamma, \nabla U(x)} - \Delta U(x) = 2v(x), \; x \in G, \quad \text{and}\quad \left(\nabla U(x)\right)'n(x) =0, \; x\in \partial G. 
\en
The solution to this Poisson equation exists due to Theorem 4.16 in \cite{BL}.

\bigskip
Now, fix a $t >0$. We apply a forward-backward martingale decomposition. By It\^o's rule applied to the semimartingale $Z$, we get
\eq\label{ito1} 
M(t) := U\left(Z(t)\right) - U\left( Z(0) \right) - \int_0^t\gen U(Z(s)) ds -  \int_0^t \left(\nabla U(Z(s))\right)' \left( N + Q \right) d L(s), 
\en
is the final element of the martingale $\nabla U(Z) \cdot W$ in time $[0,t]$. 

When we reverse time, applying It\^o's rule to $\hat Z$ we get
\eq\label{ito2}
\hat M(t) := - U\left(Z(t)\right) + U\left( Z(0) \right) - \int_0^t\gendual U(Z(s)) ds -  \int_0^t \left(\nabla U(Z(s))\right)' \left( N -Q \right) d L(s)
\en
is the final element of another martingale.

Adding \eqref{ito1} and \eqref{ito2} and using \eqref{uandv} we get
\eq\label{fbmgle}
\frac{M(t) + \hat M(t)}{2} =  \int_0^t v\left( Z(s) \right) ds.
\en
Thus, for any convex function $\phi$ we have
\eq\label{convexbnd}
E \phi\left( \int_0^t v\left( Z(s) \right) ds  \right) \le \frac{1}{2}\left[  E \phi\left( M(t) \right) + E \phi\left( \hat M(t) \right)  \right].
\en
\bigskip

\noindent\textbf{Connection to the Poisson equation.}

\begin{lemma}\label{neumannlemma}
Consider the Poisson equation with Neumann boundary condition in \eqref{neumannu}. 
\begin{enumerate}
\item[(i)] If $\norm{\nabla U}_p:= \int_G \norm{\nabla U(x)}^p \rho(x) dx < \infty$ for some $p \ge 1$, then 
\[
E \left(\frac{1}{t} \int_0^t v\left( Z(s) \right) ds  \right)^p \le  C_p t^{-p/2} \norm{\nabla U}_p,
\]
where $C_p$ is some universal constant. 
\medskip

\item[(ii)] If $\norm{\nabla U}_\infty:=\sup_{x\in G} \norm{\nabla U(x)} < \infty$, then, for any $\lambda \in \rr$, we get 
\[
E \exp\left(  \frac{\lambda}{t} \int_0^t v(Z(s)) ds \right) \le \exp\left(  \frac{\lambda^2 \norm{\nabla U}^2_\infty}{2t}  \right).
\]
Hence, for any $r > 0$, we have
\[
P\left(  \frac{1}{t}\int_0^t v(Z(s)) ds \ge r  \right) \le \exp\left(  -\frac{r^2 t}{2\norm{\nabla U}^2_\infty} \right). 
\]
\end{enumerate}
\end{lemma}

\begin{proof} We start by estimating $\iprod{M}(t)$ and $\iprod{\hat M}(t)$. We will only consider $\iprod{M}$ since the other case is symmetric. By It\^o's rule, we get
\[
\iprod{M}(t) = \int_0^t \norm{\nabla U(Z(s))}^2 ds. 
\]
Moreover, for any $k\ge 1$, Jensen's inequality implies
\[
\begin{split}
E \iprod{M}^k(t) &=t^k E \left(  \frac{1}{t} \int_0^t \norm{\nabla U(Z(s))}^2 ds  \right)^k \le t^k  E \left(  \frac{1}{t} \int_0^t \norm{\nabla U(Z(s))}^{2k} ds  \right)\\
&= t^k \frac{1}{t} \int_0^t E \norm{\nabla U(Z(s))}^{2k}ds= t^k \int_G \norm{\nabla U(x)}^{2k} \rho(x)dx= t^k \norm{\nabla U}_{2k}.
\end{split}
\]
The second last equality is due to the fact that the process is running in equilibrium. 

Suppose now $\norm{\nabla U}_{p} < \infty$ 
for some $p \ge 1$. 
We invoke Burkholder-Davis-Gundy inequality (see \cite[p.~166]{KS}) for continuous local martingales to infer
\[
E \abs{M(t)}^{p} \le C_p E \iprod{M(t)}^{p/2} \le C_p t^{p/2} \norm{\nabla U}_p.
\]
The same bound for $E \abs{\hat M(t)}^p$ proves (i) by \eqref{convexbnd}.  

Now assume that $\norm{\nabla U}_\infty < \infty$. Then, by It\^o's rule, it follows that almost surely $\iprod{M(t)} \le \norm{\nabla U}^2_\infty t$. Hence, by an exponential martingale inequality we get that for any $\lambda \in \rr$, we have
\[
E \exp\left(  \lambda M(t) \right) \le \exp\left( \frac{\lambda^2 \norm{\nabla U}^2_\infty t}{2} \right).
\]
Again obtaining the same bound for $\hat{M}$ and combining with \eqref{fbmgle}, 
we obtain 
the first part of (ii). The rest follows by a standard application of Markov's inequality and optimizing over $\lambda$.  
\end{proof}

\noindent\textbf{Poincar\'e inequality for the symmetric process.} For a large class of functions $v$, one can show that the corresponding $\norm{\nabla U}$ is in $\mathbf{L}^2(G;\rho)$. To this end, we prove first the Poincar\'e inequality in the setting of Theorem \ref{skewadd}.

\begin{lemma}\label{poincaresigma}
Let the conditions \eqref{deltaassump} and \eqref{sigmaassump} hold true. Then for any $f$ in the domain of $\mathcal{E}^\sigma$ in \eqref{symmform} which satisfies $\int_G f(x) \rho(x)\; dx=0$, we have
\[
\int_G f^2(x) \exp(\langle\tilde{\gamma},x\rangle)\; dx \le C_P \int_G \norm{\nabla f}^2(x) \exp(\langle\tilde{\gamma},x\rangle)\; dx.  
\] 
Hereby, $C_P$ is given by $\frac{4\tilde{\lambda}_n}{\min_{1\leq k\leq n-1}\tilde{\alpha}_k^2}$ with $\tilde{\lambda}_n$ being the largest eigenvalue of $\Xi^{-1}$ and $\tilde{\alpha}_k$, $k=1,\dots,n-1$ being defined as in the introduction.
\end{lemma}

\noindent\textit{Proof}.
 As remarked in the introduction, condition \eqref{deltaassump} implies that the process of spacings $Y$ possesses a unique invariant distribution, so that the density function $\rho$ can be normalized to a probability density function, which we will refer to as $\rho_1$.

We now fix a function $f$ as in the statement of the theorem, define a function $g:(\rr_+)^{n-1}\rightarrow\rr$ by $g(y)=f(\Sigma^{-1}y)$ and denote the invariant distribution of the process of spacings $Y$ by $\nu$ as before. From the definitions of $\nu$ and $\rho_1$ we see 
\eq\label{transf}
\int_G f^2(x) \rho_1(x)\;dx=\int_{(\rr_+)^{n-1}} g^2(y)\;d\nu(y).
\en
By Theorem 2 in \cite{IPBKF} the probability measure $\nu$ is a product of Exponential distributions with parameters $\tilde{\alpha}_k$, $k=1,\dots,n-1$, so that as in the derivation of Lemma \ref{exppoincare} we conclude that $\nu$ satisfies the Poincar\'e inequality with the Poincar\'e constant $\frac{4}{\min_{1\leq k\leq n-1}\tilde{\alpha}_k^2}$. In particular, it holds
\eq\label{poincare_nu}
\int_{(\rr_+)^{n-1}} g^2(y)\; d\nu(y)\leq \frac{4}{\min_{1\leq k\leq n-1}\tilde{\alpha}_k^2}\int_{(\rr_+)^{n-1}} \norm{\nabla g}^2(y)\; d\nu(y).
\en

In addition, employing the chain rule as in the previous subsection 
(see \eqref{DFcompare}), 
we can bound $\norm{\nabla g}^2(y)$ from above by $\tilde{\lambda}_n\norm{\nabla f}^2(\Sigma^{-1}y)$ for all $y\in(\rr_+)^{n-1}$, where $\tilde{\lambda}_n$ is the largest eigenvalue of $(\Sigma\Sigma')^{-1}=\Xi^{-1}$. Putting together \eqref{transf}, \eqref{poincare_nu} and the latter observation, 
we obtain 
Lemma 7. \ep 
\bigskip

\noindent\textbf{A variance bound for additive functionals.} It is well-known (see Proposition 2.1 in the notes by Bakry \cite{Bk}) that a Poincar\'e inequality as proved in Lemma \ref{poincaresigma} implies that for any function $f\in \mathbf{L}^2(G;\rho)$ such that $\int_G f(x) \rho(x) dx=0$ one has
\eq\label{semigpbnd}
\int_G \left( P_t f(x) \right)^2 \rho(x) dx \le e^{-2t/C_P} \int_G f^2(x) \rho(x)dx.
\en
Here $(P_t)_{t\geq0}$ is the Markov semigroup associated with the Dirichlet form $\mathcal{E}^\sigma$ 
in \eqref{symmform}. 
We show below that the Poincar\'e inequality entails a variance bound for a large class of additive functionals via Theorem \ref{skewadd}.

\begin{lemma}\label{poincarevariance}
For any $v\in \mathbf{L}^2(G;\rho)$ such that $\int_G v(x) \rho(x)dx=0$, suppose that the function
\eq \label{def of U}
U(x)=\int_0^\infty P_tv(x) dt, \quad x\in G,
\en
is in $C^2(\overline{G})$ and satisfies the Neumann boundary condition. Then $U$ is a solution of the Poisson equation \eqref{neumannu}. Moreover, if $\norm{\nabla{v}}\in \mathbf{L}^2(G;\rho)$, then $\norm{\nabla U}\in \mathbf{L}^2(G;\rho)$. 
\end{lemma}

\begin{proof}
We first claim that $U$ is finite everywhere on $G$. To see this choose $0 < \delta < 2/C_P$ and note that by the Cauchy-Schwarz inequality
\eq \label{CauchySchwarz}
\left(  \int_0^\infty P_t v(x) dt\right)^2  \le\left(  \int_0^\infty e^{-\delta t} dt \right) \left( \int_0^\infty  e^{\delta t} \left(P_t v(x)\right)^2 dt \right).
\en
The quantity on the left is infinite only if the quantity on the right is. But the right-hand side is integrable with respect to $\rho(x)dx$ by \eqref{semigpbnd} and our choice of $\delta$.

Let $\mathcal{A}$ denote the generator of the semigroup $(P_t)_{t\geq0}$. Then, by the usual semigroup calculus it follows that
\[
\mathcal{A} U(x) = \int_0^\infty \mathcal{A} P_t v(x) dt = \int_0^\infty \left( \frac{d}{dt} P_tv(x) \right) dt 
= - v(x).
\]
The generator $\mathcal{A}$ is an extension of the differential operator 
$(\gen + \gendual)/2$ 
over $C^2$-functions satisfying the Neumann boundary condition. It follows that $U$ is a solution of the Poisson problem with the Neumann boundary condition.  

For any two points $x,y \in G$, let $X^x, X^y$ be two processes with transition kernels $(P_t)_{t\geq0}$ driven by the same Brownian motion $B$ while starting from $x$ and $y$, respectively. 
By an application of It\^o's rule 
\[
\abs{X_t^x - X_t^y} = 2 \int^{t}_{0} \langle X_{s}^{x} - X_{s}^{y}, n(X_{s}^{x}) \rangle d \ell^{x}_{s} + 2 \int^{t}_{0} \langle X_{s}^{y} - X_{s}^{x}, n(X_{s}^{y}) \rangle d \ell^{y}_{s} + \abs{x-y} , 
\]
where $\ell^{x}\, $ and $\ell^{y}$ are the local time components of $X^{x}$ and $X^{y}$ at $\partial G$, respectively. Since $G$ is convex, the inner products in the integrands are nonpositive, and hence 
$\abs{X_t^x - X_t^y} \le \abs{x-y}$ with probability one (Lemma 3.1 of \cite{WY}). 
Taking $y \rightarrow x$, this shows $\norm{\nabla P_t v} \le P_t \norm{\nabla v}$ for all $t\ge 0$ (see inequality (3.4) in \cite{WY}). The rest is an application of the Cauchy-Schwarz inequality as 
in \eqref{CauchySchwarz}. 
\end{proof}

We can now put together our findings to give a proof of Theorem \ref{skewadd}.

\begin{proof}[Proof of Theorem 2] 
Part (i) is a consequence of a combination of Lemmata \ref{neumannlemma} and \ref{poincarevariance} along with the transformation back from $Z$ to $Y$. Part (ii) is obtained by putting together H\"older's inequality
\[
\begin{split}
\ev^{\kappa} \Big[ \Big ( \frac{1}{t} \int^{t}_{0} u(Y(s)) d s \Big)^{2-\varepsilon}\Big]
&=\ev^{\nu} \Big[\frac{d\kappa}{d\nu}\cdot\Big ( \frac{1}{t} \int^{t}_{0} u(Y(s)) d s \Big)^{2-\varepsilon}\Big]\\
&\leq\ev\Big[\Big(\frac{d\kappa}{d\nu}\Big)^{2/\epsilon}\Big]\cdot \ev^{\nu} \Big[ \Big ( \frac{1}{t} \int^{t}_{0} u(Y(s)) d s \Big)^{2}\Big]
\end{split}
\]
and part (i).
\end{proof}
\bigskip

\begin{rmk} It must be known that $U$ as defined in Lemma \ref{poincarevariance} is automatically in $C^2(\overline{G})$ and satisfies the Neumann boundary condition. However, at the present moment we cannot find a suitable reference.
\end{rmk}

\section{Convergence to equilibrium for a general SRBM}\label{general}
Throughout this section we consider a general semimartingale reflecting Brownian motion $Y$ in the $(n-1)$-dimensional orthant $\mathfrak S:=(\rr_+)^{n-1}$ in the sense of \cite{DW}. The set of parameters for such a process is given by a drift vector $\gamma\in\rr^{n-1}$, a covariance matrix $\Xi=(\xi_{ij})\in\rr^{(n-1)\times(n-1)}$ and a reflection vector field $\mathbf{r}$ on the boundary $\partial \mathfrak S$ of the orthant $\mathfrak S$ which is constant along each face. Intuitively, the process $Y$ behaves as a Brownian motion with drift vector $\gamma$ and covariance matrix $\Xi$ in the interior of $\mathfrak S$, and is reflected in the direction determined by the reflection vector field $\mathbf{r}$, whenever it hits the boundary $\partial \mathfrak S$. For a precise definition of the process $Y$ we refer to section 2 of \cite{DW}. In particular, the process of spacings  
\[
(X_{(2)}-X_{(1)},\dots,X_{(n)}-X_{(n-1)})
\]
between the particles in the particle system \eqref{ordered_part} is an SRBM in $\mathfrak S$ in the sense of \cite{DW} (see section 5 in \cite{IPBKF} for more details). 

We now give the proof of Theorem \ref{mainlyapunov}.

\subsection{Proof of Theorem \ref{mainlyapunov}.} \textbf{Step 1.} We start by introducing the differential operator
\eq
\gen:=\frac{1}{2}\sum_{i,j=1}^{n-1} \xi_{ij}\frac{\partial^2}{\partial y_i\partial y_j}+\langle \gamma,\nabla \rangle.
\en
We will construct a function $V:\;\mathfrak S\backslash\{0\}\rightarrow[1,\infty)$ which satisfies
\eq\label{lyapunov}
\gen V\leq -c_1 V +c_2 1_{B_d}
\en
for some positive constants $c_1$, $c_2$, $d$, whereby $B_d$ stands for the intersection of the ball of radius $d$ around the origin with the orthant $\mathfrak S$. This will be done in step 2. In step 3 we will combine ideas from section 5 in \cite{DMT} with \eqref{lyapunov} to finish the proof of the theorem.

\bigskip

\noindent\textbf{Step 2.} In order to provide a function satisfying \eqref{lyapunov} we start with the Lyapunov function $W$ constructed in \cite{DW}. The properties of $W$ which will use are the following: 
\begin{eqnarray}
&&W\in C^2(\mathfrak S\backslash\{0\}),\label{lyap1}\\
&&W(y)\rightarrow\infty,\;|y|\rightarrow\infty,\label{lyap2}\\
&&\sum_{i,j=1}^{n-1}\Big|\frac{\partial^2 W}{\partial y_i\partial y_j}(y)\Big|\rightarrow0,\;|y|\rightarrow\infty,\label{lyap3}\\
&&\exists\; c>0:\;\langle\gamma,(\nabla W)(y)\rangle\leq-c,\;y\in \mathfrak S\backslash\{0\},\;{\rm and}\label{lyap4}\\
&&\quad\quad\quad\quad\;\langle\mathbf{r}(y),(\nabla W)(y)\rangle\leq-c,\;y\in\partial \mathfrak S\backslash\{0\},\label{lyap4'}
\end{eqnarray}
Moreover $W$ is a homogeneous function of degree one. 

Since $W$ is homogeneous, all partial derivatives of the form $\frac{\partial W}{\partial y_i}$ and $\frac{\partial^2 W}{\partial y_i\partial y_j}$ are homogeneous functions of degree zero and negative one, respectively. This creates the problem that the second partial derivatives have a singularity at the origin. To get rid of this problem we will deform the function $W$ near the origin appropriately.

Let $\phi:(0, \infty) \rightarrow (0, \infty)$ be a twice continuously differentiable function such that
\begin{enumerate}
\item[(i)] the following asymptotic properties are satisfied near infinity: 
\eq\label{nearinf}
\lim_{x\rightarrow \infty} \phi'(x) =1, \qquad \lim_{x\rightarrow \infty} \phi''(x)=0,
\en 
\item[(ii)] and the following asymptotic properties are satisfied near zero:
\eq\label{nearzero}
\lim_{x\rightarrow 0+} \frac{\phi'(x)}{x} = \lim_{x \rightarrow 0+} \phi''(x)= 1.
\en
\item[(iii)] Moreover, the global condition $\phi'(x) \ge 0$, $x\in (0, \infty)$ is satisfied and the function $\phi'$ is bounded. 
\end{enumerate}

\bigskip

Now, define the function $U(y)= \phi(W(y))$ on $\mathfrak S\backslash \{0\}$ and note the following:
\[
\begin{split}
(\nabla U)(y) &= \phi'\left( W(y) \right) (\nabla W)(y),\\
\frac{\partial^2 U}{\partial y_i \partial y_j} (y) &= \phi''\left( W(y) \right) \frac{\partial W}{\partial y_i}(y) \frac{\partial W}{\partial y_j}(y) + \phi'\left( W(y) \right) \frac{\partial^2 W}{\partial y_i \partial y_j} (y). 
\end{split}
\]

Since $W$ is positively homogeneous and \eqref{nearinf} holds, there is a $d>0$ large enough such that if $y \notin B_d$, we have 
\eq\label{nearinf2}
\phi'\left( W(y) \right) > 1/2, \quad \text{and} \quad \Big|\sum_{i,j=1}^{n-1} \xi_{ij}\frac{\partial^2 U}{\partial y_i\partial y_j}(y)\Big|<c/4,
\en
where $c$ is the constant in \eqref{lyap4} and \eqref{lyap4'}. Hereby, the second inequality is a consequence of \eqref{nearinf}, the fact that all first partial derivatives of $W$ are homogeneous of degree zero and \eqref{lyap3}. 
In conjunction with \eqref{lyap4}, this also implies that 
\eq\label{nearinf3}
\iprod{\gamma, \nabla U(y)} \le -\frac{c}{2}, \quad \text{for all}\; y \notin B_d.
\en

Now, since $\nabla W$ is homogeneous of degree zero, and $\phi'\geq0$ throughout, we also obtain
\eq\label{globalsign}
\iprod{\mathbf{r}(y), \nabla U(y)} \le 0, \quad \text{for all}\; y \in \partial \mathfrak S\backslash \{0\}.
\en

Finally, each second partial derivative $\frac{\partial^2 W}{\partial y_i\partial y_j}$ is homogeneous of degree negative one. Hence, when $y$ approaches zero, we have from \eqref{nearzero} 
\eq\label{nearzero2}
\Big|\frac{\partial^2 U}{\partial y_i \partial y_j}(y)\Big| 
= O(1) + \frac{\phi'(W(y))}{W(y)}\cdot W(y)\Big|\frac{\partial^2 W}{\partial y_i \partial y_j} (y)\Big|= O(1).
\en
The last equality follows from \eqref{nearzero} and the fact that the function
\[
W(y) \Big|\frac{\partial^2 W}{\partial y_i \partial y_j} (y)\Big|
\]
is homogeneous of degree zero.

\bigskip

Now, we define $V(y)=e^{\lambda U(y)}$, $y\in \mathfrak S\backslash\{0\}$ with a positive constant $\lambda$ to be chosen later. An elementary computation shows
\[
\frac{\gen V}{V}=\lambda\langle\gamma,\nabla U\rangle
+\frac{\lambda}{2}\sum_{i,j=1}^{n-1} \xi_{ij}\frac{\partial^2 U}{\partial y_i\partial y_j} 
+\frac{\lambda^2}{2}(\nabla U)'\Xi(\nabla U)
\] 
on $\mathfrak S\backslash\{0\}$. 

From \eqref{nearinf2}, the fact that $\nabla U$ is globally bounded and \eqref{nearinf3}, it follows that choosing $\lambda>0$ small enough, we obtain the following estimates on the complement of $B_d$ in $\mathfrak S\backslash\{0\}$:
\[
\begin{split}
\frac{\gen V}{V}&=\lambda\langle\gamma,\nabla U\rangle
+\frac{\lambda}{2} \left[  \sum_{i,j=1}^{n-1} \xi_{ij}\frac{\partial^2 U}{\partial y_i\partial y_j}  
+\lambda(\nabla U)'\Xi(\nabla U)\right]\\
& \le \lambda\langle\gamma,\nabla U\rangle+ \frac{\lambda c}{4} \le - \frac{\lambda c}{4}.
\end{split}
\]
On the other hand, since $U$, its gradient, and its Hessian can bounded uniformly on $B_d$, one can choose a large enough $C>0$ such that
\[
{\gen V} \le - \frac{\lambda c}{2} V + C, \quad \text{for all}\; y \in B_d.
\]

Hence, we see that there exist positive constants $c_1$, $c_2$ such that
\eq\label{ineq_L}
\gen V\leq -c_1 V +c_2 1_{B_d}
\en
on $\mathfrak S\backslash\{0\}$.

\bigskip

\noindent\textbf{Step 3.} To be able to proceed as in section 5 of \cite{DMT}, we first show that $Y$ is irreducible with respect to the Lebesgue measure on $\mathfrak S$ and aperiodic, in the sense of the definitions in section 3 of \cite{DMT}. Indeed, it is well-known (see e.g. \cite{RW}) that for any fixed time $t>0$ the probability that $Y(t)$ will be in the interior of the orthant $\mathfrak S$ is equal to one. Thus, for any open ball $B\subset \mathfrak S$, the quantity $E\Big[\int_1^\infty 1_{\{Y(s)\in B\}}\;ds\Big]$ is positive, since a Brownian motion started in the interior of $\mathfrak S$ has a positive probability of hitting $B$ before hitting $\partial \mathfrak S$. Moreover, the same argument and the Feller property of $Y$ (see e.g. \cite{W}) imply that $Y$ is aperiodic with respect to the set $B_1$ and that $B_1$ is a small set for $Y$.   

From the latter observations we conclude that it suffices to check the condition in part (b) of Theorem 5.2 in \cite{DMT} to obtain our Theorem \ref{mainlyapunov} (note that we cannot use part (c) of the same theorem directly, since we do not know that $V$ belongs to the domain of the generator of $Y$ and that it is a Lyapunov function with respect to the latter). To this end, it suffices to show that there is a function $\xi:\;(0,1]\rightarrow[0,\infty)$ with $\xi(1)<1$, a petite set $B$ of the process $Y$ and a constant $c_3>0$ such that
\eq\label{drift_cond1}
\int_{\mathfrak S} V(z) P_t(y,dz)\leq\xi(t)V(y)+c_3 1_B(y)
\en
holds for all $t\in(0,1],\;y\in \mathfrak S\backslash\{0\}$. 

To show \eqref{drift_cond1}, we take the expectation in the change of variables formula in the form of equation (3.1) in \cite{W2}, and use the inequality \eqref{globalsign} to obtain
\eq\label{drift_cond2}
E\Big[e^t V(Y(t))\Big]\leq V(y)+E\Big[\int_0^t e^s(\gen V)(Y(s))\;ds\Big],
\en
whereby we start the process $Y$ in $y$. Note that the integrability of the terms inside the latter two expectations follows from the fact that the functions $V$ and $\gen V$ are bounded on compact sets and grow at most at an exponential rate as $|y|$ tends to infinity. 

To finish the proof, we plug in \eqref{ineq_L} into \eqref{drift_cond2} and use Fubini's Theorem to conclude
\eq\label{drift_cond3}
\int_{\mathfrak S} V(z) P_t(y,dz)\leq e^{-t}V(y)+c_2\int_0^t e^{s-t}P_s(y,B_R)\;ds.
\en
It remains to observe that our inequality \eqref{drift_cond3} is precisely the inequality (33) of \cite{DMT}, so that we can obtain the desired inequality \eqref{drift_cond1} from the inequality \eqref{drift_cond3} in exactly the same way as part (d) of Theorem 5.1 in \cite{DMT} is obtained from the inequality (33) there. \ep   

\bigskip

We now attempt to explicitly describe a Lyapunov function for certain rank-based models. Consider the process of spacings $(X_{(2)}-X_{(1)},\dots,X_{(n)}-X_{(n-1)})$ in the particle system \eqref{ordered_part}. It is known (see Corollary 2 in \cite{IPBKF}) that under the conditions
\begin{eqnarray}\label{condlyap}
&&\delta_1\geq\delta_2\geq\dots\geq\delta_n,\\
&&\delta_i>\delta_{i+1}\;{\rm for\;some}\;i\in\{1,\dots,n\}  
\end{eqnarray}
the process of spacings is a positive recurrent SRBM. This includes, in particular, the so-called \textit{Atlas model} introduced in \cite{F}, in which $\delta_1>0$, $\delta_2=\ldots=\delta_n=0$, $\sigma_1=\ldots=\sigma_n=1$. Under the above conditions we can provide an explicit function $V$, which may serve as a Lyapunov function in the proof of Theorem \ref{mainlyapunov}. Indeed, picking a vector $v\in \mathfrak S$ and setting $V_0(y)=e^{\langle v,y\rangle}$, $y\in \mathfrak S$ one has $V_0\geq1$ and
\[
(\gen V_0)(y)=\Big(\frac{1}{2}\sum_{i,j=1}^{n-1} \xi_{ij} v_i v_j + \sum_{i=1}^{n-1} \gamma_i v_i\Big)V_0(y),\;y\in \mathfrak S.
\] 
This shows that, if one can find a $v\in \mathfrak S$ with $\langle\gamma, v\rangle<0$ and $\langle \mathbf{r}(y), v\rangle\leq 0$, $y\in\partial \mathfrak S$, then for a $c>0$ small enough the function $V(y)=e^{c\langle v,y\rangle}$, $y\in \mathfrak S$ will satisfy
\eq
(\gen V)(y)\leq -c_1 V(y),\;y\in \mathfrak S,\quad \langle\mathbf{r}(y),(\nabla V)(y)\rangle\leq0,\;y\in\partial \mathfrak S
\en 
for some constant $c_1>0$ and 
may serve as a Lyapunov function in the proof of Theorem \ref{mainlyapunov}. 

In general, the existence of such a vector $v$ can be easily verified via the Farkas's Lemma (see e.g. \cite[page 200]{R}) which states the following. There is a vector $v$ with $\iprod{\gamma, v} < 0$ and $\iprod{\mathbf{r}(y), v} \le 0$, $y\in\partial \mathfrak S$ if and only if there is \textit{no solution} to $Rx + \gamma =0$ in $\mathfrak S$. Hereby, $R$ is the matrix whose columns are the vectors of reflection, which in our case is given by \eqref{whatisr}. Since $R$ is of full rank, we infer that there is a $v$ satisfying our requirements for a Lyapunov function if and only if $-R^{-1}\gamma\notin \mathfrak S$. 

Under our condition \eqref{condlyap} one verifies for $n\geq4$ that the vector $v=(v_1,\dots,v_{n-1})$ with components
\[
v_i=\Big(\frac{n}{2}-1\Big)^2-\Big(\frac{n}{2}-i\Big)^2+\epsilon,\;i=1,\dots,n-1
\]
has the desired properties for any $0<\epsilon<\Big(\frac{n}{2}-1\Big)^2-\Big(\frac{n}{2}-2\Big)^2$ by using the explicit formula for the reflection matrix $R$ given in \eqref{whatisr} and the concavity of the function $x\mapsto-\Big(\frac{n}{2}-x\Big)^2$. For many other examples of linear Lyapunov functions see \cite{C}. 


\section{Applications}
\subsection{Portfolio performance}\label{portfolio}
For the rest of the paper we set $\sigma_i=1$, $i\in I$ without further notice. 
As an application of Theorem 1 we shall consider 
an abstract equity market model 
$(S_{i}(\cdot),\; i \in I)$ where the market capitalization 
$S_{i}(t)$ of company $i \in I$ 
at time $t \ge 0$ is given by $S_{i}(t) := \exp (X_{i}(t))$. That is, 
each $X_{i}(t)$ in (\ref{ranksde}) gives the logarithmic capitalization 
of company $i \in I$ at time $t \ge 0$. This market model was introduced
by Fernholz in the book \cite{F} and investigated further in the articles \cite{BFK},
\cite{FI}, \cite{pp}, \cite{CP} and \cite{IPBKF} among others.

By It\^o's formula and (\ref{ranksde}) we have 
\[
d S_{i}(t) = S_{i}(t) \big( \sum_{j \in I} \delta_{j} \cdot 1\{X_{i}(t) = X_{(j)}(t) \} 
d t + dW_{i}(t) \big) + \frac{1}{2}S_{i}(t) dt , 
\]
\[
d \big( \sum_{i \in I} S_{i}(t) \big) = \sum_{i \in I}  S_{i}(t) \Big( 
\sum_{j \in I} \delta_{j} \cdot 1\{X_{i}(t) = X_{(j)}(t) \} + \frac{1}{2}\Big)
d t + \sum_{i \in I} S_{i}(t) dW_{i}(t)  \, . 
\]
Another application of It\^o's formula shows that the market capitalization weights 
$\mu_{i}(t) := S_{i}(t) / \sum_{j \in I} S_{j}(t) $, $ i \in I$  
satisfy 
\[
d \mu_{i}(t) = \big( \text{ finite variation part } \big) + \mu_{i}(t) \sum_{j \in I} \big( 
\delta_{ij} - \mu_j(t) \big) d W_{j}(t) , \quad i \in I\, , 
\]
where $\delta_{ij}$ is the Kronecker delta. Hence, the corresponding cross variation
processes grow at the rates 
\eq \label{eq: cros var mu}
\frac{d\langle \mu_{i}, \mu_{j} \rangle }{d t}(t) = 
\mu_{i}(t) \mu_{j}(t) \sum_{k \in I} \big( \delta_{ik} - \mu_k(t) \big) \big(\delta_{jk} - \mu_k(t) \big), \quad (i, j) \in I^{2}.
\en

For notational simplicity we will from now on write $D_{i}$ for the partial derivative with respect to the $i$-th variable and $D_{ij}$ for the second partial derivative with respect to the $i$-th and the $j$-th variables. 

A {\it portfolio} $({\mathbf \pi}_i(\cdot),\; i \in I)$ in the market above is defined as a stochastic process adapted to the Brownian filtration and such that 
\[
{\mathbf \pi}_{1}(t) + \cdots  + {\mathbf \pi}_{n}(t) = 1,\quad t\geq0.
\]
Its {\it value process} $V^{{\mathbf \pi}} (\cdot)$ is defined as a solution to 
\[
\frac{d V^{{\mathbf \pi}}(t)}{V^{{\mathbf \pi}}(t)} = \sum_{i \in I} {\mathbf \pi}_{i}(t) 
\frac{d S_{i}(t)}{S_{i}(t)} . 
\]
For example, ${\mathbf \pi}_{i}(.) = \mu_{i}(.)$, $i \in I$ gives the {\it market portfolio} and we will write $V^{\mu}(\cdot)$ for its value process. 
  
Following Fernholz \cite{F}, we introduce the family $\mathcal C$ of {\it portfolio generating functions} ${\bf G}: U \to (0, \infty)$, which are functions of class $C^{2}$ on some open neighborhood $U\subset\rr^n$ of the simplex 
\[
\Delta^{n}_{+}:= \{ x = (x_{1}, \ldots , x_{n}) \in [0,1]^{n}: x_{1} + \cdots + x_{n} = 1\}
\]
and such that the mapping $x \to x_{i} \log {\bf G}(x)$ is bounded on $\Delta^{n}_{+}$ for all $i \in I$. According to Theorem 3.1.5 of \cite{F},  
given a portfolio generating function ${\bf G} \in {\mathcal C}$, the functionally generated portfolio $({\mathbf \pi}_{i}(.),\;i\in I)$, defined by 
\[
{\mathbf \pi}_{i}(t) = \Big( D_{i} \log {\bf G}(\mu(t)) + 1 - \sum_{j \in I} \mu_{j} 
D_{j} \log {\bf G}(\mu(t)) \Big) \mu_{i}(t),\;\quad t\geq0,\;i\in I
\]
has a value process $V^{{\mathbf \pi}}(\cdot)$ which satisfies the master formula 
\eq \label{eq: master port}
\log (V^{{\mathbf \pi}}(t) / V^{\mu}(t)) = \log ({\bf G}(\mu(t))/{\bf G}(\mu(0))) 
+  \int^{t}_{0 }{\mathfrak g}(s) ds, \;\quad t\geq0, 
\en
where the drift part is 
\eq \label{eq: drift}
{\mathfrak g}(t) = \frac{-1 }{2 {\bf G}(\mu(t))} \sum_{(i,j) \in I^{2}} D_{ij} {\bf G}(\mu(t)) 
\frac{d\langle \mu_{i},  \mu_{j}\rangle}{dt} (t),\;\quad t\geq0. 
\en
Substituting (\ref{eq: cros var mu}) into this formula, we obtain 
${\mathfrak g}(\cdot)$ as a function of the market weights 
$\mu_{i}(\cdot)$, $i \in I$: 
\[
{\mathfrak g}(\cdot) = \frac{-1 }{2 {\bf G}(\mu(\cdot))} \sum_{(i,j) \in I^{2}} D_{ij} {\bf G}(\mu(\cdot)) 
 \mu_{i}(\cdot) \mu_{j}(\cdot) \sum_{k \in I} \big( \delta_{ik} - \mu_k(\cdot) \big) \big(\delta_{jk} - \mu_k(\cdot) \big). 
\]
The relative value process $V^{{\mathbf \pi}}(\cdot)/V^{\mu}(\cdot)$ is determined by the drift process ${\mathfrak g}$. 

We shall now rewrite the drift process ${\mathfrak g}$ in terms of the spacings process $Y$ of Theorem \ref{mainadd}. The ranked market weights $\mu_{(1)}(t) \le \mu_{(2) }(t) \le \cdots \le \mu_{(n)}(t)$ at a given time $t\geq0$ are obtained from the corresponding value of the spacings process by the formula
\[
\mu_{(j)}(t) = {\mathbf M}_{j}(Y(t)), \quad j \in I, 
\]
where the functions $ {\mathbf M}_{j} : [0, \infty)^{n-1} \to (0,1)$ are defined by 
\eq
\begin{split}\label{eq: Mj}
&{\mathbf M}_{1}(y_{1}, \ldots , y_{n-1}) := [1+ e^{y_{1}} + e^{y_{1} +y_{2} } + \cdots + e^{y_{1} + \cdots y_{n-1}}]^{-1}\, ,\\ 
&{\mathbf M}_{k}(y_{1}, \ldots , y_{n-1}) := e^{y_{1}+\cdots + y_{k-1}}  {\mathbf M}_{1}(y_{1}, \ldots , y_{n-1})\, , \quad  k = 2, \ldots , n. 
\end{split}
\en
For notational simplicity we will write ${\mathbf M}(Y(t))$ or $ M(t)$ for the value of the process $(\mu_{(1)}(\cdot), \ldots , \mu_{(n)}(\cdot))$ at a given time $t \ge 0$. Let us also introduce the partitions 
\[
\rr^n=\bigcup_{i\in I} Q^{(i)}_l,\quad l\in I,\qquad \rr^n=\bigcup_{l\in I} Q^{(i)}_l,\quad i\in I 
\]
such that for every $(i,l)\in I^2$ and every $x=(x_1,\dots,x_n)\in Q^{(i)}_l$ the component $x_i$ is the $l$-th ranked from the bottom in the set $\{x_1,\dots,x_n\}$.  

We recall that the domain of the portfolio generating function ${\bf G}$ is given by $\Delta_{+}^{n}$ rather than the set 
\eq
\Delta_{+,\leq}^{n}:=\{(x_1,\dots,x_n)\in \Delta_{+}^{n}:\;x_{1} \le x_{2} \le \dots \le x_{n}\}. 
\en
However, if the function ${\mathbf G}$ is permutational invariant, in the sense that
\eq \label{eq: perm inv} 
{\bf G}(x_{1}, \ldots , x_{n}) = {\bf G}(x_{p(1)}, \dots x_{p(n)})
\en
for every permutation $(p(1), \ldots , p(n))$ of $I$, then we may and will view ${\mathbf G}$ as a function on $\Delta_{+,\leq}^{n}$. Next, let us introduce the functions
\eq \label{eq: gjk}
g_{jk}(x) := \sum_{(h,i) \in I^{2}} D_{hi} {\bf G}(x) \cdot 1_{Q_{j}^{(h)}}(x) \cdot 1_{Q_{k}^{(i)}} (x) , \quad (j,k) \in I^{2}
\en
defined on the simplex $\Delta_{+}^{n}$.

Important examples of portfolio generating functions are: 
\begin{itemize}
\item ${\bf G}(x) = ( \sum_{i \in I} x_{i}^{p})^{1/p}$ for some $0 < p < 1 $ 
(diversity) or
\item ${\bf G}(x) = 1-\frac{1}{2}\sum_{i\in I}(x_i-n^{-1})^2$ (quadratic Gini coefficient) or
\item ${\bf G}(x) = (1-p)^{-1} \log ( \sum_{i \in I} x_{i}^{p})$ for some $p \neq 1$ (R\'enyi entropy) or
\item ${\bf G}(x) = - \sum_{i \in I} x_{i} \log x_{i}$ (entropy) or
\item ${\bf G}(x) = (x_{1} \cdots x_{n})^{1/n}$ (equal-weighting generating function).
\end{itemize}
Under condition \eqref{eq: perm inv} (in particular, in the latter examples) each process $g_{jk}(\mu(t))$, $t\geq0$ can be rewritten as  
\eq \label{eq: gjk tilde}
\widetilde{g}_{jk} ( {\mathbf M}(Y(t))),\; t\geq0
\en
for appropriate functions $\widetilde{g}_{jk}$ on $\Delta_{+,\leq}^{n}$ and all $(j,k)\in I^2$.  Thus, in this case, the drift process
${\mathfrak g}$ in (\ref{eq: drift}) can be written as 
\begin{equation*}
\begin{split} 
&{\mathfrak g}(t) = \frac{-1}{2{\bf G}(\mu(t))} 
\sum_{(h,i) \in I^{2}} D_{hi} {\bf G}(\mu(t)) \mu_{h}(t) \mu_{i}(t) 
\sum_{\ell \in I} \big( \delta_{hl} - \mu_l(t) \big) \big(\delta_{il} - \mu_l(t) \big) \\
&= \frac{-1}{2{\bf G}(\mu(t))} 
\sum_{(h,i,j,k) \in I^{4}} 1_{Q_{j}^{(h)}\cap Q_{k}^{(i)}}  D_{hi} {\bf G}(\mu(t))
 \mu_{h}(t) \mu_{i}(t) 
\sum_{\ell \in I} \big( \delta_{hl} - \mu_l(t) \big) \big(\delta_{il} - \mu_l(t) \big)\\
&= \frac{-1}{2{\bf G}(\mu(t))} 
\sum_{(j,k)\in I^{2}} g_{jk}(\mu(t))  \mu_{(j)}(t) \mu_{(k)}(t) 
\sum_{\ell \in I} \big( \delta_{jl} - \mu_{(l)}(t) \big) \big(\delta_{kl} - \mu_{(l)}(t) \big)\\
&=\widetilde{u}(Y(t)),\;t\geq0, 
\end{split} 
\end{equation*}
where the function $\widetilde{u}:[0,\infty)^{n-1}\rightarrow\rr$ is defined by 
\[\widetilde{u}(y) := - \frac{1}{2{\bf G}( {\mathbf M}(y))} 
\sum_{(j,k) \in I^{2}} \widetilde{g}_{jk}( {\mathbf M}(y)) {\mathbf M}_{j}(y) {\mathbf M}_{k}(y) 
\sum_{\ell \in I} \big( \delta_{jl} - {\mathbf M}_{l}(y) \big) \big(\delta_{kl} - {\mathbf M}_{l}(y) \big).
\]
Finally, using the stationary distribution $\nu$ of the spacings process described in the introduction, we define a normalized version $u:[0,\infty)^{n-1}\rightarrow\rr$ of $\tilde{u}$ by   
\eq \label{eq: u}
\, u (y) := \widetilde{u}(y) - \nu(\widetilde{u}) = \widetilde{u}(y) - \int_{[0,\infty)^{n-1}} \tilde{u}(z)\;d\nu(z).   
\en

Combining Theorem 1, (\ref{eq: master port}) and the representation of the drift process ${\mathfrak g}$ in terms of the spacings process $Y$, we can compare the value process $V^{{\mathbf \pi}}$ of the portfolio generated by a permutation invariant function ${\bf G} \in \mathcal C$ to the value process $V^{\mu}$ of the market portfolio with the same initial value. Similarly, instead of (\ref{eq: master port}), we can use 
\[
\log (V^{\mu}(t) / V^{\pi}(t)) = - \log ({\bf G}(\mu(t) ) / {\bf G}(\mu(0)) ) + \int^{t}_{0} 
(-{\mathfrak g}(s)) ds 
\]
and apply Theorem 1 with $-u$. All in all, we obtain the following corollary. 

\begin{cor}\label{port_perf} Suppose that ${\bf G}$ is permutation invariant, and that $u$ in $(\ref{eq: u})$ satisfies $0 < \text{Var}_{\nu}(u) = \sigma^{2} < \infty$, $ \lVert u \rVert_{\infty} \in (0, \infty)$ and $ \delta^{2}(u) > 0$. Then for every initial spacing distribution 
$\kappa$ of $Y(0)$ the ratios between the portfolio value $V^{{\mathbf \pi}}(\cdot)$ generated by ${\mathbf G}$ and the market portfolio value $V^{\mu}(\cdot)$ satisfy the estimates
\[
\begin{split}
&P(V^{{\mathbf \pi}}(t) / V^{\mu} (t) \ge c_{1}^{+}(t){\bf G}(\mu(t)) /{\bf G}(\mu(0)) ) \le \Big\lVert \frac{d\kappa}{d\nu} \Big\rVert_{2} 
\exp \Big[ - \frac{t}{\beta} c_{2} \Big] ,\\
&P(V^{{\mathbf \mu}}(t) / V^{\pi} (t) \ge c_{1}^{-}(t){\bf G}(\mu(0)) /{\bf G}(\mu(t)) ) \le \Big\lVert \frac{d\kappa}{d\nu} \Big\rVert_{2} 
\exp \Big[ - \frac{t}{\beta} c_{2} \Big]   
\end{split}
\]
for all $t,r,\epsilon>0$, where $c_{1}^{\pm}(t):= \exp [\{r\pm \nu(\widetilde{u})\}t]  $, $\beta$ is given by \eqref{whatisbeta} and 
\[
c_{2} := \max \Big( \frac{r^{2}}{\delta^{2}(u)} , 4 \varepsilon (\varepsilon + \sigma^{2}) \Big( \sqrt{1 + \frac{r^{2}}{2 \varepsilon (\varepsilon + \sigma^{2})^{2} \lVert u \rVert_{\infty}^{2}}} -1\Big) \Big) . 
\] 
\end{cor}

\begin{exm}
We now examine the statement of Corollary \ref{port_perf} in the cases of the portfolios generated by the diversity, the quadratic Gini coefficient, the R\'enyi entropy, the entropy and the equal-weighting generating function defined above.  

For the {\it diversity-weighted portfolio}, that is, the portfolio generated by the diversity, we compute
\[
(D_{jk}{\bf G})(x)=(1-p)\Big(\sum_{i\in I} x_i^p \Big)^{(1/p)-2}x_j^{p-1}x_k^{p-1}+\delta_{jk}(p-1)\Big(\sum_{i\in I} x_i^p \Big)^{(1/p)-1}x_j^{p-2},
\] 
for all $(j,k)\in I^2$. In addition, we note that the vector ${\bf M}(Y(t))$, which we will abbreviate by $M(t)$ from now on, satisfies $M_1(t)<M_2(t)<\ldots<M_n(t)$ for Lebesgue almost every $t\geq0$ with probability one. Indeed, this is a consequence of the fact that the set 
\[
\{t\geq0:\;X_{(i)}(t)=X_{(j)}(t)\;\text{for\;some\;}1\leq i<j\leq n\}
\]
is a Lebesgue zero set with probability one (see \cite{pp} for more details). Hence, we may conclude $\tilde{g}_{jk}(M(t))=(D_{jk}{\bf G})(M(t))$ for Lebesgue almost every $t\geq0$ with probability one. Using the formula above for the partial derivatives $D_{jk}{\bf G}$, $(j,k)\in I^2$ and taking into account the latter observation, we obtain after some elementary computations:
\[
\tilde{u}(Y(t))=\frac{1-p}{2}\Big(1-\frac{\sum_{i\in I} (M_i(t))^{2p}}{(\sum_{i\in I} (M_i(t))^p)^2}\Big)
\]
for Lebesgue almost every $t\geq0$ almost surely. Moreover, since by definition the market weights $M_i(t)$, $i\in I$ are $[0,1]$-valued, we conclude using the Cauchy-Schwarz inequality that $\tilde{u}(Y(t))$ takes values in $\Big[0,\frac{(n-1)(1-p)}{2n}\Big]$ for Lebesgue almost every $t\geq0$ with probability one. Hence, the fluctuations of the relative performance of the diversity-weighted portfolio with respect to the market portfolio are controlled by the estimates of Corollary \ref{port_perf} with $\|u\|_\infty$ replaced by $\frac{(n-1)(1-p)}{2n}$ and $\delta(u)$ replaced by $\frac{(n-1)(1-p)}{n}$. 

\bigskip

For the {\it portfolio generated by the quadratic Gini coefficient} a similar computation yields
\[
(D_{jk}{\bf G})(x)=-\delta_{jk},\;(j,k)\in I^2
\]
and, hence, 
\[
\tilde{u}(Y(t))=\frac{\sum_{i\in I} (M_i(t))^2-2(\sum_{i\in I} (M_i(t))^3) + (\sum_{i\in I} (M_i(t))^2)^2}{2- \sum_{i\in I}(M_i(t)-n^{-1})^2}
\]
for Lebesgue almost every $t\geq0$ almost surely. Moreover, the inequality $(M_i(t))^2-2(M_i(t))^3+(M_i(t))^4\geq0$ for every $t\geq0$ and $i\in I$, and the fact that the market weights are $[0,1]$-valued and sum up to one imply:  
\[
0\leq \tilde{u}(Y(t))\leq \frac{2}{2-(1-n^{-1})^2}
\]
for Lebesgue almost every $t\geq0$ with probability one. Thus, the inequalities of Corollary \ref{port_perf} apply with $\|u\|_\infty$ replaced by $\frac{2}{2-(1-n^{-1})^2}$ and $\delta(u)$ replaced by $\frac{2}{1-\frac{1}{2}(1-n^{-1})^2}$.

\bigskip

In the case of the {\it portfolio generated by the R\'enyi entropy} we get 
\[
(D_{jk}{\bf G})(x)=\frac{p^2}{p-1}\cdot\frac{x_j^{p-1}x_k^{p-1}}{(\sum_{i\in I} x_i^p)^2}-\delta_{jk}\frac{px_j^{p-2}}{\sum_{i\in I} x_i^p}
\] 
for all $(j,k)\in I^2$. An analogous computation to the case of the diversity-weighted portfolio yields here: 
\[
\begin{split}
\tilde{u}(Y(t))= &\frac{p}{2\log(\sum_{i\in I} M_i(t)^p)}\\
&\cdot\Big(1-p+p\cdot\frac{\sum_{i \in I} (M_i(t))^{2p}}{(\sum_{i \in I} (M_i(t))^{p})^2} 
- 2\cdot\frac{\sum_{i \in I} (M_i(t))^{p+1}}{\sum_{i \in I} (M_i(t))^p}+\sum_{i\in I} M_i(t)^2\Big)
\end{split}
\]
for Lebesgue almost every $t\geq0$ almost surely. Although the values of the process $M_n(t)$, $t\geq0$ can be arbitrarily close to one, a simple analysis based on L'H\^opital's rule shows that the values of the process $|\tilde{u}(Y(t))|$, $t\geq0$ are uniformly bounded for Lebesgue almost every $t\geq0$ with probability one, so that Corollary \ref{port_perf} applies in this case as well. 

In the case of the {\it entropy-weighted portfolio}, that is, the portfolio generated by the entropy, one computes
\begin{eqnarray*}
&&(D_{jk}{\bf G})(x)=-\delta_{jk}\frac{1}{x_j},\;(j,k)\in I^2,\;{\rm and}\\
&&\tilde{u}(Y(t))=\frac{1-\sum_{i\in I} (M_i(t))^2}{-2\sum_{i\in I} M_i(t)\log M_i(t)}
\end{eqnarray*}
for Lebesgue almost every $t\geq0$ almost surely. Moreover, the estimate
\[
|\tilde{u}(Y(t))|\leq\frac{1-(M_n(t))^2}{-2M_n(t)\log M_n(t)},
\]
an analysis of the right-hand side of the latter inequality as $M_n(t)$ approaches one, and the inequality $M_n(t)\geq n^{-1}$ show that the values of the process $|\tilde{u}(Y(t))|$, $t\geq0$ are uniformly bounded for Lebesgue almost every $t\geq0$ with probability one. Hence, our Corollary \ref{port_perf} can be also applied in this case. 

\bigskip

Finally, for the {\it equal-weighted portfolio}, that is, the portfolio generated by the equal-weighting generating function, one easily checks 
\begin{eqnarray*}
&&(D_{jk}{\bf G})(x)=\frac{1}{nx_jx_k}(x_1\cdots x_n)^{1/n}\Big(\frac{1}{n}-\delta_{jk}\Big),\;(j,k)\in I^2,\;{\rm and}\\
&&\tilde{u}(Y(t))=\frac{n-1}{n}
\end{eqnarray*}
for Lebesgue almost every $t\geq0$ almost surely. Thus, this is a trivial case and, although Corollary \ref{port_perf} applies, it does not give a meaningful estimate. 

\end{exm}

\subsection{Fluctuations of the market weights}\label{marketweightsetc} On pages 46 and 52 of their survey on Stochastic Portfolio Theory \cite{FI}, Fernholz and Karatzas pose the following open questions which we restate slightly for the models considered in this article. Consider an abstract rank-based equity market model with $n$ companies as defined in the previous subsection and consider for any given time $t\geq0$ the ranked market weights:
\[
\mu_{(1)}(t) \le \mu_{(2)}(t) \le \cdots \le \mu_{(n)}(t).
\]  

\noindent What can one say about the following objects:
\begin{enumerate}
\item[(i)] approximate laws of $\mu_{(1)}(t)$ and $\mu_{(n)}(t)$, 
\item[(ii)] fluctuations of the moving averages
\[
\frac{1}{T} \int_0^T \mu_{(k)}(t)dt,\quad k=1,\ldots,n.
\]
\item[(iii)] In addition, the following question is of interest: What is the approximate deviation from $1/n$ of the quantity
\[
\frac{1}{T} \left\{ \text{amount of time the $i$-th market weight has rank $k$ during $[0,T]$}  \right\}.
\]
\end{enumerate}

Our estimates on the rate of convergence to equilibrium will allow us to partially answer each of these questions. 

\bigskip

To answer question (ii), we recall from the previous subsection that at any given time each ranked market weight can be written as a time-independent continuous bounded function of the vector of spacings. Thus, our Theorem \ref{mainadd} can be applied directly to the moving averages in question (ii). It gives completely explicit estimates on the fluctuations of the latter, provided that one can compute the first two moments of the corresponding ranked market weight in equilibrium. This is in general a daunting task. In the following theorem we provide formulas for all moments of the ranked market weights in the Atlas model under their stationary distribution. Although not explicitly numerical, they can be effectively computed via a software such as Mathematica. This also gives a partial answer to question (i) for the Atlas model. Recall that the latter is the special case of the particle system in \eqref{ranksde} with $\delta_1=\delta>0$, $\delta_2=\ldots=\delta_n=0$, $\sigma_1=\ldots=\sigma_n=1$. 

\begin{thm} 
Consider the Atlas model with $\delta$ being the drift of the lowest ranked particle. In equilibrium, the law of the ranked market weights is determined by the following Laplace transform:
\eq\label{laplacemuk}
\tau(\theta):=E \left[\exp\left(-\frac{\theta}{\mu_{(k)}}\right)\right]= e^{-\theta} \left(  \phi(\theta)\right)^{n-k} 
E[\psi_{\overline{\beta}}(\theta)^{k-1}], 
\en
where
\begin{enumerate}
\item[(i)] $\phi$ is the Laplace transform of $e^W$ with $W$ being an Exponential random variable of parameter $2\delta / n$,
\item[(ii)] $\psi_{\overline{\beta}}$ is the conditional Laplace transform of 
\[
 \left( \frac{\overline{\beta}}{(1-\overline{\beta})V+\overline{\beta}}  \right)^{n/2\delta} , 
\]
conditional on the value of $\overline{\beta}$, where $\overline{\beta}$ is Beta$(n-k+1,k)$ distributed, $V$ is uniformly distributed on $(0,1)$, and $\overline{\beta}$, $V$ are independent. 
\end{enumerate}
In particular, in equilibrium, we obtain all moments of $\mu_{(k)}$ by the formula
\[
E\big[(\mu_{(k)})^{r}\big] = \frac{1}{(r-1)!} \int_0^\infty \theta^{r-1} \tau(\theta) d\theta, \qquad r=1,2,\ldots.
\]
\end{thm}

\begin{proof}
We fix a $k\in\{1,\dots,n\}$ and recall the following result from Pal and Pitman \cite[Theorem 8]{pp}. Let $\xi_1, \xi_2, \ldots, \xi_{n-1}$ be independent Exponential random variables with respective parameters 
\[
\frac{2\delta}{n} \left( n - i  \right), \qquad i=1,2,\ldots, n-1.
\]

Then take $\xi_0$ to be any random variable and set $\eta_i:=\xi_0+ \xi_1 + \ldots + \xi_{i-1}$, $i=1,2,\ldots,n$. Then, the following equality in law holds:
\[
\mu_{(k)} = \frac{e^{\eta_k}}{\sum_{j=1}^n e^{\eta_j} }, \qquad k=1,2,\ldots, n.
\]
Note that $\xi_0$ does not play any role, since it gets cancelled in the latter fraction. 

Thus,
\eq\label{decompmuk}
\begin{split}
\frac{1}{\mu_{(k)}}&= \frac{\sum_{j=1}^{k-1} e^{\eta_j}}{e^{\eta_k}} + 1 + \frac{\sum_{j=k+1}^n e^{\eta_j}}{e^{\eta_k}}=  A_k + 1 + \Sigma_k.
\end{split}
\en
Hereby,
\[
A_k = \sum_{j=1}^{k-1} \exp\Big(-\sum_{l=j}^{k-1} \xi_l \Big), \quad \Sigma_k= \sum_{j=k+1}^n \exp\Big( \sum_{l=k}^{j-1} \xi_l  \Big)
\]
are independent random variables. 

\bigskip

Next, let $\vartheta_1, \vartheta_2, \ldots, \vartheta_{m}$ denote i.i.d. Exponential random variables with some parameter $\alpha$. Then the R\'enyi representation of the order statistics of i.i.d. Exponential random variables states that the random variables 
\[
\vartheta_{(i+1)} - \vartheta_{(i)}, \quad i=1,2,\ldots, m-1,
\]
are independent and Exponentially distributed with respective parameters $\alpha(m-i)$, $i=1,2,\ldots, m-1$. Hereby, $\vartheta_{(1)}\leq\vartheta_{(2)}\leq\ldots\leq\vartheta_{(m)}$ are the order statistics of the vector $(\vartheta_1, \vartheta_2, \ldots, \vartheta_{m})$. We shall use this representation to express $A_k$ and $\Sigma_k$ in a symmetric way. 

Now, set $m=n-k$, $\alpha=2\delta/n$ and define $\Delta_i:= \vartheta_{(i+1)} - \vartheta_{(i)}$, $i=1,2,\ldots, m-1$, $\Delta_0:=\vartheta_{(1)}$. Then, we get the following equality in distribution:
\[
\Delta_i = \xi_{k+i}, \quad 0\le i\le n-k-1.
\]
Hence, it holds  
\[
\sum_{i=1}^{n-k} e^{\vartheta_i} = \sum_{i=1}^{n-k} e^{\vartheta_{(i)}} = \sum_{i=1}^{n-k}\exp\Big(\sum_{l=0}^{i-1}\Delta_l\Big)= \Sigma_k
\] 
in distribution. Thus, $\Sigma_k$ is the sum of $(n-k)$ i.i.d. random variables. In particular,
\eq\label{sigmaklaplace}
E \left(e^{-\theta \Sigma_k}\right) = \left( \phi_\alpha(\theta)\right)^{n-k},
\en
where $\phi_{\alpha}$ is the Laplace transform of $e^{\vartheta_1}$ given by
\eq
\phi_{\alpha}(\theta)=\int_0^\infty\alpha  \exp\big(-\alpha x - \theta e^x\big)\;dx. 
\en

\bigskip

The case of $A_k$ is a bit more convoluted. First, let $T_1, T_2, \ldots, T_n$ be i.i.d. Exponential random variables with parameter $\alpha=2\delta/n$, which are independent of the $\vartheta_i$'s. Setting $R_j=\exp(-T_j)$ for $j=1,\ldots,n$, it is clear that each random variable $R_j$ is distributed according to the Beta distribution Beta($\alpha,1$). Hence, we can write $R_j=U_j^{1/\alpha}$, $j=1,\ldots,n$ with suitable i.i.d. uniformly on $(0,1)$ distributed random variables $U_1, U_2, \ldots, U_n$. 

Now, using the R\'enyi representation again we obtain the following identity in distribution:
\eq\label{abyu}
A_k = \sum_{j=1}^{k-1} \frac{e^{-T_{(k)}}}{e^{-T_{(j)}}} = \sum_{j=1}^{k-1} \frac{R_{(n-k+1)}}{R_{(n-j+1)}}
=\sum_{j=1}^{k-1} \left( \frac{U_{(n-k+1)}}{U_{(n-j+1)}}  \right)^{1/\alpha}=\sum_{j=n-k+2}^{n} \left( \frac{U_{(n-k+1)}}{U_{(j)}}  \right)^{1/\alpha},
\en 
where $T_{(1)}\le T_{(2)}\le\ldots\le T_{(n)}$, $R_{(1)}\le R_{(2)}\le\ldots\le R_{(n)}$ and $U_{(1)}\le U_{(2)}\le\ldots\le U_{(n)}$ are the order statistics of the vectors $(T_1,T_2,\ldots,T_n)$, $(R_1,R_2,\ldots,R_n)$ and $(U_1,U_2,\ldots,U_n)$, respectively. 

We now employ some known identities related to the Uniform distribution. First, we note that the vector
\[
\left( U_{(1)}, U_{(2)} - U_{(1)}, U_{(3)} - U_{(2)}, \ldots, U_{(n)} - U_{(n-1)}, 1 - U_{(n)} \right)
\]
is distributed uniformly over the $(n+1)$-simplex $\{x\in \rr^{n+1}:\; x_i\ge 0, \; \sum_i x_i=1  \}$, i.e. as Dirichlet$(1,1,\ldots,1)$.

By the aggregation rule for the Dirichlet distribution the vector
\[
\left( U_{(n-k+1)}, U_{(n-k+2)} - U_{(n-k+1)}, \ldots, U_{(n)} - U_{(n-1)}, 1 - U_{(n)} \right)
\]
has the Dirichlet$(n-k+1, 1,\ldots,1)$ distribution on the $(k+1)$-simplex. 

Hence, by the usual Beta-Gamma algebra, we see that
\[
\frac{1}{1-U_{(n-k+1)}}\left( U_{(n-k+2)} - U_{(n-k+1)}, \ldots, U_{(n)} - U_{(n-1)}, 1- U_{(n)} \right)
\]
is distributed as Dirichlet$(1,1,\ldots,1)$ over the $k$-simplex independently of $U_{(n-k+1)}$, which is distributed as Beta$(n-k+1, k)$.

As a corollary, taking partial sums, we deduce that the law of the vector 
\[
\left( \frac{U_{(n-k+2)}-U_{(n-k+1)}}{1-U_{(n-k+1)}}, \frac{U_{(n-k+3)}-U_{(n-k+1)}}{1-U_{(n-k+1)}}, \ldots, \frac{U_{(n)}-U_{(n-k+1)}}{1-U_{(n-k+1)}}  \right)
\]
is the same as that of the order statistics of $(k-1)$ i.i.d. Uniform$(0,1)$ random variables $V_1, \ldots, V_{k-1}$ independent of 
$\overline{\beta}:=U_{(n-k+1)}$.   

Using the expression in \eqref{abyu} we obtain that $A_k$ has the same law as
\[
\sum_{j=1}^{k-1} \left( \frac{\overline{\beta}}{(1-\overline{\beta})V_j+\overline{\beta}}  \right)^{1/\alpha}.
\]
Hence,
\[
E\left( e^{-\theta A_k}  \right) = 
E[\psi_{\overline{\beta}}(\theta)^{k-1}],
\]
where 
$ \psi_{\overline{\beta}} $
is the conditional 
Laplace transform of $$\left(\frac{\overline{\beta}}{(1-\overline{\beta})V_1+ \overline{\beta}}\right)^{1/\alpha}$$
conditioned on $\overline{\beta}$. 

\bigskip

Hence, from \eqref{decompmuk} we get
\[
\tau(\theta)=E\left( e^{-\theta/\mu_{(k)} }  \right) =e^{-\theta} \phi_\alpha(\theta)^{n-k} E[\psi_{\overline{\beta}}(\theta)^{k-1}] ,
\]
which leads to \eqref{laplacemuk}.
\bigskip

To find the moments, we use the following fundamental identity: 
for any $r > 0$,
we have
\[
\frac{1}{\Gamma(r)}\int_0^\infty \theta^{r-1} e^{-\theta/\mu} d\theta = \mu^{r}.
\]
Replacing $\mu$ by $\mu_{(k)}$ above and interchanging expectation and integral we get
\[
E[(\mu_{(k)})^{r}] = \frac{1}{\Gamma(r)} \int_0^\infty \theta^{r-1} \tau(\theta) d\theta.
\] 
This completes the proof.
\end{proof}
\bigskip

Unfortunately, the beautiful identities provided by the Atlas model do not extend to more general models. Asymptotic derivation of moments (when $n$ tends to infinity) is possible in certain regimes due to an approximation by the atoms of the Poisson-Dirichlet distribution. Please see the article by Chatterjee and Pal \cite{CP} for the details. Formulas for moments in the two-parameter Poisson-Dirichlet model can be found in Pitman and Yor \cite[Proposition 17]{PY}.

\bigskip

We now give an answer to question (iii) in the case of the particle system in \eqref{ranksde} under the condition \eqref{deltaassump}. Recall that we assume $\sigma_i=1$, $i\in I$ throughout.

\begin{thm}
Let 
\[
\tilde{X}(t)=\Big(X_1(t)-n^{-1}\sum_{i\in I} X_i(t),\dots,X_n(t)-n^{-1}\sum_{i\in I} X_i(t)\Big),\;t\geq0
\]
be the centered version of the particle system in \eqref{ranksde} and assume that \eqref{deltaassump} holds. Then the process $\tilde{X}$ is Markovian and possesses a unique invariant distribution $\tilde{\nu}$. Moreover, for every measure $\kappa$ which is absolutely continuous with respect to $\tilde{\nu}$ and such that $\frac{d\kappa}{d\tilde{\nu}}$ is square integrable with respect to $\tilde{\nu}$, one has for all $t,r,\epsilon>0$ the estimate
\begin{eqnarray*}
&&P\left( \frac{1}{t} \int_0^t u(\tilde{X}(s)) ds \ge r  \right)\le \\
&&\Big\|\frac{d\kappa}{d\tilde{\nu}}\Big\|_2 \exp\left[ -\frac{t}{C_{P}}
\max\left( \frac{r^2}{\delta^2(u)}, 4\epsilon(\epsilon+ \sigma^2)\left( \sqrt{1 + \frac{r^2}{2\epsilon(\epsilon + \sigma^2)^2 \norm{u}_\infty^2}} -1   \right) \right)  \right]
\end{eqnarray*}
for all bounded measurable functions $u$ provided that the initial value $\tilde{X}(0)$ is distributed according to $\kappa$, $\tilde{\nu}(u)=0$ and $Var_{\tilde{\nu}}(u)=\sigma^2$. Hereby, $C_P$ is a positive constant depending only on $n$ and $\delta_1,\dots,\delta_n$ (see \eqref{cent_poin_const} for an explicit expression).

In particular, the latter estimate holds for functions of the form 
\[
u(x)=1_{\{x_i=x_{(j)}\}}-n^{-1},\quad (i,j)\in I^2 
\]
with $\sigma^2=n^{-2}(n-1)$, $\delta(u)=1$ and $\norm{u}_\infty=1-n^{-1}$. 
\end{thm}

\noindent{\it Proof.} 1) The Markov property and the existence and uniqueness of the invariant distribution of the process $\tilde{X}$ were shown in Theorem 8 of \cite{pp}. Thus, we only need to prove the inequality in the statement of the theorem. To this end, we introduce for each vector $x\in\rr^n$ a permutation $\pi(x)$ of the set $\{1,\dots,n\}$ such that $x_{\pi(x)(1)}\leq\dots\leq x_{\pi(x)(n)}$ holds. Since the process $\tilde{X}$ is a diffusion process with state space
\[
\mathfrak{H}=\{x\in\rr^n:\;x_1+\dots+x_n=0\},
\]
it is a Feller process, the space $C^\infty_c(\mathfrak{H})$ is a core for its generator $\gen$ and on that space the generator is given by    
\eq
(\gen f)(x)=\frac{1}{2}\sum_{i,j=1}^n a_{ij}\frac{\partial^2 \tilde{f}}{\partial x_i\partial x_j}(x) 
+\sum_{i=1}^n \tilde{\mu}_{\pi(x)^{-1}(i)}\frac{\partial \tilde{f}}{\partial x_i}(x),\quad x\in \mathfrak{H}, 
\en
where $a_{ij}=\delta_{ij}-n^{-1}$, $\tilde{\mu}_i=\delta_i-n^{-1}\sum_{j\in I} \delta_j$ and $\tilde{f}$ is the composition of the projection of vectors in $\rr^n$ onto $\mathfrak{H}$ and $f$ (see chapter 18 of \cite{Ka} for the details). 

Next, we define the cone
\[
\mathfrak{H}_{\leq}=\{x\in \mathfrak{H}:\;x_1\leq\dots\leq x_n\}
\] 
and introduce the mapping $\Theta:\,\rr^n\rightarrow\rr^n$, which arranges the coordinates of a vector $x\in\rr^n$ in ascending order, as well as the mapping 
\[
\Phi:\;\mathfrak{H}_{\leq}\rightarrow(\rr_+)^{n-1},\quad x\mapsto(x_2-x_1,\ldots,x_n-x_{n-1}),
\] 
which maps vectors in $\mathfrak{H}_{\leq}$ to the corresponding vectors of spacings. We recall from Theorem 8 in \cite{pp} the following facts. The invariant distribution $\tilde{\nu}$ of the process $\tilde{X}$ is absolutely continuous with respect to the Lebesgue measure on $\mathfrak{H}$, its density is proportional to $e^{-\sum_{k=1}^{n-1}\Phi_k(\Theta(x))\alpha_k}$ and the process $\tilde{X}$ is reversible with respect to $\tilde{\nu}$. 

\medskip

\noindent 2) In view of the results in step 1, as well as Theorem 3.1 in \cite{GLWY} it suffices to show that the Poincar\'e inequality 
\eq\label{cent_poin}
\int_{\mathfrak{H}} f(x)^2 e^{-\sum_{k=1}^{n-1}\Phi_k(\Theta(x))\alpha_k} \;dx
\leq C_P \int_{\mathfrak{H}} (-\gen f)(x)f(x)e^{-\sum_{k=1}^{n-1}\Phi_k(\Theta(x))\alpha_k}\; dx
\en
holds for a suitable constant $C_P>0$, whereby the integration is performed with respect to the Lebesgue measure on the hyperplane $\mathfrak{H}$. To this end, let $\Sigma$ be a positive definite symmetric $n\times n$-matrix such that $\Sigma^2=(a_{ij})_{1\leq i,j\leq n}$. Then the same computation as on the top of page 64 in \cite{BBCG}, but with $\nabla$ replaced by $\Sigma\nabla$, shows that inequality \eqref{cent_poin} is fulfilled provided that there exists a function $V:\;\mathfrak{H}\rightarrow[1,\infty)$ which belongs to $H^1(\mathfrak{H};\tilde{\nu})$ and satisfies
\eq\label{cent_lyap}
\frac{\gen V}{V}\leq -\frac{1}{2C_P} 
\en
almost everywhere on $\mathfrak{H}$. Hereby, $H^1(\mathfrak{H};\tilde{\nu})$ is the space of square integrable functions with respect to $\tilde{\nu}$, whose gradient exists in the weak sense and is square integrable with respect to $\tilde{\nu}$. 

We claim that there is a $0<c<1$ such that the function
\eq
V(x)=e^{(c/2)\sum_{k=1}^{n-1}\Phi_k(\Theta(x))\alpha_k} 
\en
defined on $\mathfrak{H}$ has the desired properties. Indeed, $V$ is a Lipschitz function and, thus, differentiable almost everywhere. Moreover, the condition $0<c<1$ shows that $V$ belongs to $H^1(\mathfrak{H};\tilde{\nu})$. In addition, it holds $\sum_{k=1}^{n-1}\Phi_k(\Theta(x))\alpha_k\geq0$ by definition, so that we have $V\geq1$. Finally, we compute
\begin{eqnarray*}
\frac{(\gen V)}{V}(x)&=&\frac{c^2}{8}\sum_{i,j=1}^n a_{\pi(x)(i)\pi(x)(j)} (\alpha_{i-1}-\alpha_i)(\alpha_{j-1}-\alpha_j)\\
&&+\frac{c}{2}\sum_{i=1}^n \tilde{\mu}_{\pi(x)^{-1}(i)}(\alpha_{\pi(x)^{-1}(i)-1}-\alpha_{\pi(x)^{-1}(i)})\\
&=&\frac{c^2}{2}\sum_{i,j=1}^n a_{ij}\tilde{\mu}_i\tilde{\mu}_j-c\sum_{i=1}^n \tilde{\mu}_{\pi(x)^{-1}(i)}^2
\leq\Big(\frac{c^2 \lambda_{max}}{2} -c\Big)\|\tilde{\mu}\|^2,
\end{eqnarray*}
where we have set $\alpha_0=\alpha_n=0$ and have written $\lambda_{max}$ for the maximal eigenvalue of the matrix $(a_{ij})_{1\leq i,j\leq n}$. It is not hard to see that the eigenvalues of the latter are given by $1,\ldots,1,0$, so that $\lambda_{max}=1$. Thus, for any $c\in(0,1)$ the Poincar\'e inequality \eqref{cent_poin} with the constant $C_P=-\frac{1}{(c^2-2c)\|\tilde{\mu}\|^2}$ holds true. Taking the limit $c\uparrow1$, we conclude that the Poincar\'e inequality \eqref{cent_poin} is satisfied with  
\eq\label{cent_poin_const}
C_P=\frac{1}{\|\tilde{\mu}\|^2}=\frac{1}{\sum_{i\in I} (\delta_i-n^{-1}\sum_{j\in I}\delta_j)^2}.
\en
This finishes the proof. \ep

\section*{Acknowledgement}

The authors are grateful to Zhen-Qing Chen and Ruth Williams for pointing out several references. They also thank Ioannis Karatzas for multiple useful discussions.

\bibliographystyle{alpha}

\begin{thebibliography}{50}

\bibitem{AA}
\textsc{Arguin, L.~-P.} and \textsc{Aizenman, M.} (2009). 
On the structure of quasi-stationary competing particles systems. 
\textit{Ann. Probab.} \textbf{37} 1080-1113.

 \bibitem{Bk}
 \textsc{Bakry, D.} (2002)
 Functional inequalities for Markov semigroups. Tata Institute. Available at \texttt{http://hal.archives-ouvertes.fr/docs/00/35/37/24/PDF/tata.pdf}.


\bibitem{BBCG} 
\textsc{Bakry, D.}, \textsc{Barthe, F.}, \textsc{Cattiaux P.}, and \textsc{Guillin A.} (2008)
A simple proof of the Poincar\'e inequality for a large class of probability measures including the log-concave case.
\textit{Elect. Comm. in Probab.} \textbf{13}, 60--66. 

\bibitem{BFK} 
\textsc{Banner, A.}, \textsc{Fernholz, R.}, and \textsc{Karatzas, I.} (2005)
Atlas models of equity markets. 
\textit{Ann. Appl. Probab.} \textbf{15} 2296--2330.

\bibitem{BG}
\textsc{Banner, A.} and \textsc{Ghomrasni, R.} (2008).
Local times of ranked continuous semimartingales. 
\textit{Stochastic Process. Appl.} \textbf{118} 1244-1253.



\bibitem{BW}
\textsc{Barthe, F.} and \textsc{Wolff, P.} (2009).
Remarks on non-interacting conservative spin systems: The case of gamma distributions.
\textit{Stochastic Processes and their Applications} \textbf{119}, 2711--2723.

\bibitem{BP}
\textsc{Bass, R.} and \textsc{Pardoux, E.} (1987). 
Uniqueness for diffusions with piecewise constant coefficients. 
\textit{Probab. Theory Relat. Fields} \textbf{76} 557-572.
 
 
 
 \bibitem{B}
 \textsc{Bramson, M.} (2011)
 A positive recurrent reflecting Brownian motion with divergent fluid path. 
\textit{Ann. Appl. Probab.} \textbf{21} 951-986. 
 
 
\bibitem{BL}
\textsc{Budhiraja, A.} and \textsc{Lee, C.}  (2006)
Long time asymptotics for constrained diffusions in polyhedral domains. 
\textit{Stochastic processes and their applications.} \textbf{117}, 1014 -- 1036.  
 
\bibitem{BCP} 
\textsc{Burdzy, K.}, \textsc{Chen, Z.-Q.}, and \textsc{Pal, S.} (2011)
Archimedes' principle for Brownian liquid. 
To appear in \textit{Ann. Appl. Probab.}.

\bibitem{CP}
\textsc{Chatterjee, S.} and \textsc{Pal, S} (2010)
\newblock A phase transition behavior for Brownian motions interacting through their ranks. 
\textit{Probab. Theory Relat. Fields} \textbf{147} (1-2), 123--159. 

\bibitem{chatpal2}
\textsc{Chatterjee, S.} and \textsc{Pal, S.} (2008).
\newblock A combinatorial analysis of interacting diffusions. 
To appear in \textit{J. Theor. Probab.}.

\bibitem{C}
\textsc{Chen, H.} (1996)
A sufficient condition for the positive recurrence of a semimartingale reflecting Brownian motion in an orthant. 
\textit{Ann. Appl. Probab.} \textbf{6} (3), 758--765.

\bibitem{DMT}
\textsc{Down, D.}, \textsc{Meyn S. P.}, and \textsc{Tweedie, R. L.} (1995)
Exponential and Uniform Ergodicity of Markov Processes.
\textit{Ann. Probab.} \textbf{23} 1671-1691.

\bibitem{DW}
\textsc{Dupuis, P.} and \textsc{Williams, R. J.} (1994)
Lyapunov functions for semimartingale reflecting Brownian motions.
\textit{Ann. Probab.} \textbf{22} 680-702.


\bibitem{EKBTY}
\textsc{El Kharroubi, A.}, \textsc{Ban Tahar, A.}, and \textsc{Yaacoubi, A.} (2002)
On the stability of the linear Skorokhod problem in an orthant. 
\textit{Math. Meth. Oper. Res.} \textbf{56}, 243--258.


\bibitem{F}
\textsc{Fernholz E. R.} (2002).
\textit{Stochastic Portfolio Theory}.
Springer, New York.

\bibitem{FI}
\textsc{Fernholz, R.} and \textsc{Karatzas, I.} (2009)
Stochastic Portfolio Theory: an Overview. 
In: Bensoussan, A., Zhang, Q. (eds.) \textit{Handbook of Numerical Analysis: Volume XV: Mathematical Modeling and Numerical Methods in Finance}, pp. 89�167. North Holland, Oxford.


\bibitem{GLWY}
\textsc{Guillin, A.}, \textsc{L\'eonard, C.}, \textsc{Wu, L.}, and \textsc{Yao, N.} (2009)
Transportation-information inequalities for Markov processes. 
\textit{Probab. Theory Relat. Fields} \textbf{144} (3-4), 669--695.


\bibitem{harris65}
\textsc{Harris, T. E.} (1965).
\newblock Diffusion with ``collisions'' between particles.
\newblock {\em J. Appl. Probab.} \textbf{2} 323-338.



\bibitem{HR} \textsc{Harrison, J. M.} and \textsc{Reiman, M. I.} (1981). 
Reflected Brownian motion in an orthant. 
\textit{Ann. Probab.} \textbf{9} 302-308.

\bibitem{HW} \textsc{Harrison, J. M.} and \textsc{Williams R. J.} (1987).
Multidimensional reflected Brownian motions having Exponential stationary distributions.
\textit{Ann. Probab.} \textbf{15} 115--137.

\bibitem{IK} \textsc{Ichiba, T.} and \textsc{Karatzas, I.} (2009). 
On collisions of Brownian particles. \textit{Ann. Appl. Probab}. \textbf{20} 951--977.

\bibitem{IPBKF} 
\textsc{Ichiba, T.}, \textsc{Papathanakos, V.}, \textsc{Banner, A.}, \textsc{Karatzas, I.}, and \textsc{Fernholz, R.} (2010). 
Hybrid Atlas Models. 
\textit{Ann. Appl. Probab.} \textbf{21} 609--644.  

\bibitem{joumal}
\textsc{Jourdain, B.} and \textsc{Malrieu, F.} (2008).
\newblock Propagation of chaos and Poincar\'{e} inequalities for a system of particles interacting through their cdf.
\newblock \textit{Ann. Appl. Probab.} \textbf{18} 1706-1736. 


\bibitem{Ka}
\textsc{Kallenberg, O.} (1997). 
\textit{Foundations of modern probability.}
Springer, New York.

\bibitem{KS}
\textsc{Karatzas, I.} and \textsc{Shreve, S.} (1991)
\textit{Brownian Motion and Stochastic Calculus. Second Edition.}
Graduate Texts in Mathematics, Springer.  

\bibitem{KO}
\textsc{Kouachi, S.} (2008). 
Eigenvalues and Eigenvectors of Some Tridiagonal Matrices with non Constant diagonal entries. 
\textit{Appl. Math. (Warsaw)} \textbf{35} 107-120.

\bibitem{L}
\textsc{Ledoux, M.} (2000)
\textit{The Concentration of Measure Phenomenon}.
Mathematical Surveys and Monographs \textbf{89}, American Mathematical Society.


\bibitem{sheppmckean}
\textsc{McKean, H. P.} and \textsc{Shepp, L.} (2005).
\newblock The advantage of capitalism vs. Socialism depends on the criterion.
\newblock Available at {\em www.emis.de/journals/ZPOMI/v328/p160.ps.gz}.

\bibitem{N}
\textsc{Nagasawa, M.} (1961)
The adjoint process of a diffusion with reflecting barrier.
\textit{Kodai Math. Sem. Rep.} \textbf{13} 235-248.

\bibitem{pp} \textsc{Pal, S.} and \textsc{Pitman J.} (2008)
One-dimensional Brownian particle systems with rank-dependent drifts. 
\textit{Ann. Appl. Probab.} \textbf{18} 2179-2207.

\bibitem{PS}
\textsc{Pal, S.} and \textsc{Shkolnikov, M.} (2011)
Concentration of measure for systems of Brownian particles interacting through their ranks. 
Preprint available at \texttt{arXiv:1011.2443v1}. 


\bibitem{PY}
\textsc{Pitman, J.} and \textsc{Yor, M.} (1997)
The two-parameter Poisson-Dirichlet distribution derived from a stable subordinator. 
\textit{Ann. Probab.} \textbf{25} (2), 855--900. 


\bibitem{racz}
\textsc{R\'acz, M.Z.} (2010)
Competing prices: Analyzing a stochastic interacting particle system. 
Diploma thesis. Available at \texttt{http://www.stat.berkeley.edu/~racz/RaczMiklosThesis.pdf}.


\bibitem{RW} \textsc{Reiman, M. I.} and \textsc{Williams, R. J.} (1988)
A boundary property of semimartingale reflecting Brownian motions.
\textit{Probab. Theory Relat. Fields} \textbf{77} 87-97. 

\bibitem{R}
\textsc{Rockafellar, R. T.} (1979)
\textit{Convex Analysis}. Princeton University Press. 


\bibitem{ruzaizenman}
\textsc{Ruzmaikina, A.} and \textsc{Aizenman, M.} (2005).
Characterization of invariant measures at the leading edge for competing particle systems.
\textit{Ann. Probab.}, \textbf{33} (1), 82-113.


\bibitem{shkol}
\textsc{Shkolnikov, M.} (2009). Competing Particle Systems Evolving by I.I.D. Increments. 
\textit{Electron. J. Probab.} \textbf{14} 728-751. 

\bibitem{sh} \textsc{Shkolnikov, M.} (2010). Competing particle systems evolving by interacting Levy processes. Preprint available at http://arxiv.org/abs/1002.2811. To appear in \textit{Ann. Appl. Probab.}

\bibitem{sh2} \textsc{Shkolnikov, M.} (2010). Large systems of diffusions interacting through their ranks. Preprint available at http://arxiv.org/abs/1008.4611. 

\bibitem{Sw} \textsc{Swanson, J.} (2007). Weak convergence of the scaled median of independent Brownian motions. \textit{Probab. Theory Relat. Fields} \textbf{138} 269-304.


\bibitem{W}
\textsc{Williams, R. J.} (1987)
Reflected Brownian motion with skew symmetric data in a polyhedral domain. 
\textit{Probab. Theory Relat. Fields} \textbf{75} 459--485.

\bibitem{W2}
\textsc{Williams, R. J.} (1995)
Semimartingale reflecting Brownian motions in the orthant.
Stochastic Networks. Eds.: F. P. Kelly, R. J. Williams. Springer.


\bibitem{WY}
\textsc{Wang,  F-Y.} and \textsc{Yan L.} (2010) 
Gradient estimate on the {N}eumann semigroup and applications. 
Preprint available at http://arxiv.org/abs/1009.1965v2. 
\end{thebibliography}

\end{document}